\documentclass[a4paper,twoside]{amsart}
\usepackage[utf8]{inputenc}
\usepackage{amssymb,amsmath,amsthm,amscd}
\usepackage{mathrsfs}
\usepackage{enumerate}
\usepackage{graphicx}

\usepackage{lmodern}

\newcommand{\C}{\ensuremath{\mathbb{C}}}
\newcommand{\R}{\ensuremath{\mathbb{R}}}
\newcommand{\N}{\ensuremath{\mathbb{N}}}
\newcommand{\Z}{\ensuremath{\mathbb{Z}}}

\newcommand{\U}{\ensuremath{\mathrm{U}}}

\newcommand{\SU}{\ensuremath{{\mathrm{SU}(2)}}}
\newcommand{\SO}{\ensuremath{{\mathrm{SO}(2)}}}

\newcommand\diag{\ensuremath{\textrm{diag}}}
\newcommand\sh{\ensuremath{\sinh}}

\theoremstyle{definition}
\newtheorem{thmA}{Theorem}

\newtheorem{thm}{Theorem}[section]

\newtheorem{dfn}[thm]{Definition}
\newtheorem{lemma}[thm]{Lemma}
\newtheorem{prop}[thm]{Proposition}
\newtheorem{corollary}[thm]{Corollary}

\newtheorem{rem}[thm]{Remark}

\newtheorem{notation}[thm]{Notation}

\author{Tim de Laat}
\thanks{TdL is a Postdoctoral Fellow of the Research Foundation -- Flanders (FWO) and was partially supported by the Belgian Interuniversity Attraction Pole P07/18 and the Danish National Research Foundation through the Centre for Symmetry and Deformation (DNRF92). The research of MdlS was partially supported by ANR grants OSQPI and NEUMANN}
\address{Tim de Laat
\newline KU Leuven, Department of Mathematics,
\newline Celestijnenlaan 200B -- Box 2400, B-3001 Leuven, Belgium}
\email{tim.delaat@wis.kuleuven.be}

\author{Mikael de la Salle}
\thanks{2010 \emph{Mathematics Subject Classification}. Primary 22D12; Secondary 22E46, 46B20.}
\address{Mikael de la Salle
\newline CNRS-ENS de Lyon,
\newline UMPA UMR 5669
\newline F-69364 Lyon cedex 7, France
}
\email{mikael.de.la.salle@ens-lyon.fr}

\title{Strong property (T) for higher rank simple Lie groups}

\begin{document}
\DeclareGraphicsRule{*}{mps}{*}{}

\begin{abstract}
  We prove that connected higher rank simple Lie groups have Lafforgue's strong property (T) with respect to a certain class of Banach spaces $\mathcal{E}_{10}$ containing many classical superreflexive spaces and some non-reflexive spaces as well. This generalizes the result of Lafforgue asserting that $\mathrm{SL}(3,\mathbb{R})$ has strong property (T) with respect to Hilbert spaces and the more recent result of the second named author asserting that $\mathrm{SL}(3,\mathbb{R})$ has strong property (T) with respect to a certain larger class of Banach spaces. For the generalization to higher rank groups, it is sufficient to prove strong property (T) for $\mathrm{Sp}(2,\mathbb{R})$ and its universal covering group. As consequences of our main result, it follows that for $X \in \mathcal{E}_{10}$, connected higher rank simple Lie groups and their lattices have property (F$_X$) of Bader, Furman, Gelander and Monod, and that the expanders contructed from a lattice in a connected higher rank simple Lie group do not admit a coarse embedding into $X$.
\end{abstract}

\maketitle

\section{Introduction}
In 1967, Kazhdan introduced property (T) for groups in order to prove that certain groups are finitely generated \cite{kazhdan} (see also \cite{bekkadelaharpevalette}). A locally compact group has property (T) if the trivial representation of the group is isolated in the unitary dual of the group equipped with the Fell topology. This property, which usually plays the role of a rigidity property, has important applications in different areas of mathematics.

Over the years, several strengthenings of property (T) have been studied. In this article, we focus on Lafforgue's strong property (T), which he introduced as an obstruction to a certain approach to the Baum-Connes Conjecture. We work with a slightly more flexible notion of strong property (T), as defined by the second named author in \cite{delasalle1}. Recall that a length function on a locally compact group $G$ is a continuous function $\ell:G \to \R_+$ such that $\ell(g^{-1})= \ell(g)$ and $\ell(g_1 g_2)\leq \ell(g_1)+\ell(g_2)$ for all $g,g_1,g_2 \in G$.
\begin{dfn}\label{def=strongT} 
A locally compact group $G$ has strong property (T) with respect to a class $\mathcal{E}$ of Banach spaces, denoted by (T$^{\mathrm{strong}}_{\mathcal{E}}$), if for every length function $\ell$ on $G$ there is a sequence of compactly supported symmetric Borel measures $m_n$ on $G$ such that for every Banach space $X$ in $\mathcal E$ there is a constant $t>0$ such that the following holds: for every strongly continuous representation $\pi:G \rightarrow B(X)$ satisfying $\|\pi(g)\|_{B(X)} \leq L e^{t \ell(g)}$ for some $L\in\R_+$, the sequence $\pi(m_n)$ converges in the norm topology on $B(X)$ to a projection onto the $\pi(G)$-invariant vectors in $X$.
\end{dfn}
Lafforgue's original definition of strong property (T) in \cite{lafforguestrengthenedpropertyt} corresponds to taking $\mathcal{E}$ to be the class of Hilbert spaces. We will denote this property by (T$^{\mathrm{strong}}_{\mathrm{Hilbert}}$). The Banach space strong property (T) of \cite{lafforguefourier} is denoted by (T$^{\mathrm{strong}}_{\mathrm{Banach}}$), which corresponds to taking $\mathcal{E}$ to be the class of Banach spaces with nontrivial (Rademacher) type. By the work of Lafforgue, it is known that (T$^{\mathrm{strong}}_{\mathcal{E}}$) passes from a group to its cocompact lattices.

In \cite{lafforguestrengthenedpropertyt}, Lafforgue proved that word-hyperbolic groups do not satisfy (T$^{\mathrm{strong}}_{\mathrm{Hilbert}}$). It follows that connected simple Lie groups with real rank $1$ do not have (T$^{\mathrm{strong}}_{\mathrm{Hilbert}}$), since such a group is either locally isomorphic to $\mathrm{SO}(n,1)$ or $\mathrm{SU}(n,1)$ for some $n \geq 2$, in which case it does not have Kazhdan's property (T), or it is locally isomorphic to $\mathrm{Sp}(n,1)$ for some $n \geq 2$ or to $F_{4(-20)}$, in which case it contains a word-hyperbolic cocompact lattice.

Lafforgue also showed that $\mathrm{SL}(3,\mathbb{R})$ has (T$^{\mathrm{strong}}_{\mathrm{Hilbert}}$), with the consequence that any connected almost $\mathbb{R}$-simple algebraic group whose Lie algebra contains a copy of $\mathfrak{sl}(3,\mathbb{R})$ has (T$^{\mathrm{strong}}_{\mathrm{Hilbert}}$). In \cite{delasalle1}, the second named author proved that $\mathrm{SL}(3,\mathbb{R})$ has (T$^{\mathrm{strong}}_{\mathcal{E}_4}$). For $r > 2$, the class $\mathcal{E}_r$ (see Section \ref{subsec=bs} for the precise definition) is a certain class of Banach spaces containing the Hilbert spaces, many classical superreflexive spaces, and some non-reflexive spaces. We point out that if $r_1 \geq r_2 > 2$, then $\mathcal{E}_{r_1} \subset \mathcal{E}_{r_2}$. An open question is whether the classes $\mathcal{E}_r$ contain all spaces of nontrivial type (see Section \ref{subsec=bs} for the definition of type).

The aim of this article is to extend, in a way, these results on $\mathrm{SL}(3,\mathbb{R})$ to all connected higher rank simple Lie groups. To do this, we need to consider strong property (T) for $\mathrm{Sp}(2,\mathbb{R})$ and its universal covering group. We are able to prove that both these groups have (T$^{\mathrm{strong}}_{\mathcal{E}_{10}}$), which implies our main theorem.
\begin{thmA} \label{thm=maintheorem}
Let $G$ be a connected simple Lie group with real rank at least $2$. Then $G$ has strong property (T) with respect to the class $\mathcal{E}_{10}$.
\end{thmA}
The assertions that $\mathrm{Sp}(2,\mathbb{R})$ and its universal covering group $\widetilde{\mathrm{Sp}}(2,\mathbb{R})$ have (T$^{\mathrm{strong}}_{\mathcal{E}_{10}}$) follow from explicit decay estimates of matrix coefficients of representations with small exponential growth of these groups on Banach spaces in $\mathcal{E}_{10}$. In fact, we will consider Banach spaces satisfying certain technical conditions that are naturally satisfied by the spaces in $\mathcal{E}_{10}$. Parts of the computations in this article rely on the methods used in the work of Haagerup and the first named author on the failure of the Approximation Property for connected higher rank simple Lie groups (see \cite{haagerupdelaat1} and \cite{haagerupdelaat2}), and on the work in \cite{delaat1}.

Apart from its aforementioned relation to the Baum-Connes Conjecture, strong property (T) has two other interesting applications due to Lafforgue (see \cite{lafforguestrengthenedpropertyt} and \cite{lafforguefourier}) that we want to mention. Firstly, let us recall that for second countable locally compact groups, Kazhdan's property (T) is equivalent to Serre's property (FH). A locally compact group $G$ has property (FH) if every continuous affine isometric action of $G$ on a real Hilbert space has a fixed point. Analogues of property (T) and property (FH) in the setting of Banach spaces were introduced by Bader, Furman, Gelander and Monod in \cite{baderfurmangelandermonod}. However, unlike for the case of Hilbert spaces, these properties are not equivalent. For a Banach space $X$, a locally compact group $G$ is said to have property (T$_X$) (resp.~property ($\overline{\mathrm{T}}_X)$) if for every continuous linear isometric representation (resp.~every uniformly equicontinuous linear representation) $\rho:G \rightarrow O(X)$ , the quotient representation $\rho^{\prime}:G \rightarrow O(X / X^{\rho(G)})$ does not have almost invariant vectors. The group $G$ is said to have property (F$_X$) (resp.~property ($\overline{\mathrm{F}}_X$)) if every continuous action of $G$ on $X$ by affine isometries (resp.~every uniformly equicontinuous affine action of $G$ on $X$) has a fixed point. In \cite{baderfurmangelandermonod}, it was proved that for second countable locally compact groups, property (T) implies property (F$_X$) for certain Banach spaces, the most notable examples being $L^p$-spaces. Bader, Furman, Gelander and Monod conjectured that connected semisimple Lie groups with finite center and higher rank simple factors satisfy property ($\overline{\mathrm{F}}_X$) for every superreflexive Banach space $X$. In fact, they conjectured a stronger statement (see \cite[Conjecture 1.6]{baderfurmangelandermonod}). It was proved by Lafforgue that a locally compact group $G$ has ($\overline{\mathrm{F}}_X$) if it has (T$^{\mathrm{strong}}_{X \oplus \C}$). In fact, it has a fixed point property for affine actions with linear part having small exponential growth. In Section \ref{section=fixed_point} we obtain the following corollary of Theorem \ref{thm=maintheorem}, which supports the conjecture of Bader, Furman, Gelander and Monod.
\begin{corollary}\label{coro=fixed_point}
  Let $G$ be a connected simple Lie group with real rank at least $2$. Then $G$ and its lattices have properties (F$_X$) and ($\overline{\mathrm{F}}_X$) for every $X \in \mathcal{E}_{10}$.
\end{corollary}
The second application of strong property (T) that was found by Lafforgue is on embeddings of families of expanders (see \cite{lubotzky} and \cite[Section 3]{pisiermemams}). Let $\Gamma$ be a lattice in a connected higher rank simple Lie group, and assume that there exists a sequence $(\Gamma_i)_{i \in \mathbb{N}}$ of finite-index subgroups of $\Gamma$ such that $|\Gamma / \Gamma_i| \to \infty$ for $i \to \infty$. Note that this assumption is in particular satisfied when $\Gamma$ is residually finite. If $S$ denotes a finite symmetric generating set of $\Gamma$, and if $Y_i=\Gamma / \Gamma_i$ is the corresponding graph with natural metric denoted by $d_i$, then $(Y_i,d_i)_{i \in \mathbb{N}}$ is a sequence of expanders, since $\Gamma$ has Kazhdan's property (T). A sequence of expanders $(Y_i,d_i)$ is said to embed coarsely into a Banach space $X$ if there exist a function $\rho:\mathbb{N} \rightarrow \mathbb{R}_{+}$ such that $\rho(n) \to \infty$ for $n \to \infty$ and $1$-Lipschitz functions $f_i:Y_i \rightarrow X$ such that $\|f_i(y)-f_i(y^{\prime})\|_X \geq \rho(d_i(y,y^{\prime}))$ for all $i \in \mathbb{N}$ and $y,y^{\prime} \in X_i$. By \cite[Section 5]{lafforguestrengthenedpropertyt} (see also \cite[Section 5.2]{lafforguefourier} and \cite{delasalle1}), the following result follows immediately from Theorem \ref{thm=maintheorem}.
\begin{corollary}
  Let $\Gamma$ be a lattice in a connected simple Lie group with real rank at least $2$, and let $(Y_i,d_i)$ be a family of expanders constructed from $\Gamma$ in the way mentioned above. Then $(Y_i,d_i)$ does not admit a coarse embedding in any Banach space in $\mathcal{E}_{10}$.
\end{corollary}
In the non-Archimedean setting, a strong analogue of Theorem \ref{thm=maintheorem} is known. Indeed, for a non-Archimedean local field $F$, Lafforgue proved that $\mathrm{SL}(3,F)$ has property (T$^{\mathrm{strong}}_{\mathrm{Banach}}$) (see \cite{lafforguestrengthenedpropertyt} and \cite{lafforguefourier}), and Liao proved that $\mathrm{Sp}(2,F)$ has it as well, implying that any connected almost $F$-simple algebraic group with $F$-split rank at least $2$ has property (T$^{\mathrm{strong}}_{\mathrm{Banach}}$) (see \cite{liao}). Analogues of the corollaries above follow as well.
\begin{rem}
Let us point out that the estimate of Lemma \ref{lemma=estimate_S_theta_in_schatten_classes} and its proof (which is postponed to Appendix \ref{section=p10}) imply an improvement of the results in \cite{delaat1} and \cite{haagerupdelaat2}. In those articles it is proved that for a lattice $\Gamma$ in a connected higher rank simple Lie group and $p\in [1,\frac{12}{11})\cup(12,\infty]$, the noncommutative $L^p$-space $L^p(L(\Gamma))$ does not have the completely bounded approximation property (CBAP) or operator space approximation property (OAP). From our computations, it follows that even for $p \in [1,\frac{10}{9})\cup(10,\infty]$, the space $L^p(L(\Gamma))$ does not have these properties. It is an interesting open question whether this failure of the CBAP and OAP can be extended to all $p \neq 2$.
\end{rem}
This article is organized as follows. We recall some preliminaries in Section \ref{sec:preliminaries}. In Section \ref{section=harmonic_analysis_SU2}, we consider certain aspects of harmonic analysis on $\mathrm{SU}(2)$, which is a subgroup of $\mathrm{Sp}(2,\mathbb{R})$ and $\widetilde{\mathrm{Sp}}(2,\mathbb{R})$, postponing certain computations to Appendix \ref{section=p10}. The explicit decay results for representations of $\mathrm{Sp}(2,\mathbb{R})$ and $\widetilde{\mathrm{Sp}}(2,\mathbb{R})$ with small exponential growth is proved in Sections \ref{section=sp2} and \ref{section=sp2cov}, respectively. Our main theorem is proved in Section \ref{section=strongt}. In Section \ref{section=fixed_point}, we obtain Corollary \ref{coro=fixed_point}.

\subsection*{Acknowledgements} We thank the referee for his very thorough reading and useful comments. We thank Nicolas Monod for allowing us to add his argument for Proposition \ref{prop=pintegrable}. The second named author thanks \'Etienne Ghys for enlightening discussions on quasi-morphisms on $\widetilde{\mathrm{Sp}}(2,\mathbb{R})$ and for communicating the work \cite{BargeGhys}.

\section{Preliminaries} \label{sec:preliminaries}
\subsection{Lie groups} \label{subsec:liegroups}
A universal covering group of a connected Lie group $G$ is a Lie group $\widetilde{G}$ together with a surjective Lie group homomorphism $\sigma:\widetilde{G} \rightarrow G$ such that $(\widetilde{G},\sigma)$ is a simply connected covering space of $G$. Every connected Lie group $G$ has a simply connected covering space $\widetilde G$, which can be made into a universal covering group. Indeed, if $\sigma:\widetilde{G} \rightarrow G$ is the corresponding covering map and $\tilde{1} \in \sigma^{-1}(1)$, there exists a unique multiplication on $\widetilde{G}$ that makes $\widetilde{G}$ into a Lie group in such a way that $\sigma$ is a surjective Lie group homomorphism. Universal covering groups (of connected Lie groups) are unique up to isomorphism. Moreover, they satisfy the exact sequence $1 \rightarrow \pi_1(G) \rightarrow \widetilde{G} \rightarrow G \rightarrow 1$, where $\pi_1(G)$ denotes the fundamental group of $G$. For details, see \cite[Section I.11]{knapp}.

Let $G$ be a connected semisimple Lie group with Lie algebra $\mathfrak{g}$. The KAK decomposition of $G$ is given by $G=KAK$, where $K$ comes from a Cartan decomposition $\mathfrak{g}=\mathfrak{k} + \mathfrak{p}$ ($K$ has Lie algebra $\mathfrak{k}$) and $A$ is an abelian Lie group whose Lie algebra $\mathfrak{a}$ is a maximal abelian subspace of $\mathfrak{p}$. If the center of $G$ is finite, then $K$ is a maximal compact subgroup. The real rank of $G$ is defined as the dimension of $\mathfrak{a}$. Given a KAK decomposition $G=KAK$ and $g \in G$, it is not the case that there exist unique $k_1,k_2 \in K$ and $a \in A$ such that $g=k_1ak_2$. However, by choosing a set of positive roots and restricting to the closure $\overline{A^{+}}$ of the positive Weyl chamber $A^{+}$, we still have the decomposition $G=K\overline{A^{+}}K$. Also, if $g=k_1ak_2$, where $k_1,k_2 \in K$ and $a \in \overline{A^{+}}$, then $a$ is unique. For details, see \cite[Section IX.1]{helgasonlie}.

\subsection{The group $\mathrm{Sp}(2,\R)$} \label{subsec:psp2}
Let $I_2$ denote the $2 \times 2$ identity matrix, and let $J \in M_4(\R)$ be defined by $J=\left( \begin{array}{cc} 0 & I_2 \\ -I_2 & 0 \end{array} \right)$. The Lie group $\mathrm{Sp}(2,\R)$ is defined as the group of linear transformations of $\R^4$ that preserve the standard symplectic form $w(x,y) = \langle Jx,y\rangle$; that is,
\[
  \mathrm{Sp}(2,\R)=\{g \in \mathrm{GL}(4,\mathbb{R}) \mid g^T J g = J\},
\]
where $g^T$ denotes the transpose of $g$. Let us point out that some authors denote this group by $\mathrm{Sp}(4,\R)$. 

Let $K$ denote the maximal compact subgroup of $\mathrm{Sp}(2,\R)$ given by
\[
  K= \bigg\{ \left( \begin{array}{cc} A & -B \\ B & A \end{array} \right) \in \mathrm{M}_{4}(\mathbb{R}) \biggm\vert A+iB \in \mathrm{U}(2) \bigg\}.
\]
This group is isomorphic to $\mathrm{U}(2)$ through the isomorphism $\iota(A+iB) = \begin{pmatrix}A & -B \\ B & A\end{pmatrix}$. This formula defines a ring isomorphism $\iota$ between $M_2(\C)$ and $M_4(\R)$ satisfying $\iota(X^*) = \iota(X)^T$ and $\det(\iota(X)) = |\det(X)|^2$ for $X \in M_2(\C)$. Let $H$ denote the subgroup of $K$ given by $\iota(\mathrm{SU}(2))$. For $\beta,\gamma \in \R$, let $D(\beta,\gamma) = \diag(e^\beta,e^\gamma,e^{-\beta},e^{-\gamma}) \in \mathrm{Sp}(2,\R)$. A KAK decomposition is given by $\mathrm{Sp}(2,\R)=K\overline{A^{+}}K$, where $\overline{A^{+}}=\left\{D(\beta,\gamma) \mid \beta \geq \gamma \geq 0\right\}$.

We will denote the Lie algebra of $\mathrm{Sp}(2,\R)$ by $\mathfrak{sp}_2$, and by $\exp \colon \mathfrak{sp}_2 \to \mathrm{Sp}(2,\R)$ the corresponding exponential map. Let us also denote by $\mathfrak{k}$, $\mathfrak{h}$ and $\mathfrak{a}$ the Lie subalgebras corresponding to the subgroups $K$, $H$ and $A$, respectively.

For $s \in \R$, we define the following element of $K$:
\begin{equation} \nonumber
 v_s = \iota(e^{is} I_2).
\end{equation}
Since $v_\pi = -1$ belongs to the center of $\mathrm{Sp}(2,\R)$, the formula $t \cdot g = v_t g v_{-t}$ defines an action of $\R/\pi \Z$ on $\mathrm{Sp}(2,\R)$.

\subsection{The universal covering group of $\mathrm{Sp}(2,\R)$} \label{subsec:psp2cov}
Let $\widetilde{\mathrm{Sp}}(2,\R)$ be the universal covering group of $\mathrm{Sp}(2,\mathbb{R})$. As explained in Section \ref{subsec:liegroups}, it is a connected simple Lie group that is simply connected as a topological space, and there is a surjective Lie group homomorphism $\sigma:\widetilde{\mathrm{Sp}}(2,\R) \to \mathrm{Sp}(2,\R)$. Let us recall some facts about $\widetilde{\mathrm{Sp}}(2,\mathbb{R})$, in particular its KAK decomposition. For more details, we refer to \cite{haagerupdelaat2}, in which the structure of $\widetilde{\mathrm{Sp}}(2,\R)$ is discussed in more detail. The Lie algebra of $\widetilde{\mathrm{Sp}}(2,\R)$ is $\mathfrak{sp}_2$. Let $\widetilde \exp: \mathfrak{sp}_2 \rightarrow \widetilde{\mathrm{Sp}}(2,\R)$ be the corresponding exponential map. The group $\widetilde A=\widetilde \exp(\mathfrak{a})$ is isomorphic to $A$, the group $\widetilde H = \widetilde \exp(\mathfrak{h})$ is isomorphic to $H$, and $\widetilde \exp( \R J)$ is a $\Z$-covering group of $\exp(\R J) = \{ v_t \mid t \in \R/{2\pi\Z}\}$. For $a \in A$ (resp.~$h \in H$), we denote by $\widetilde a$ (resp.~$\widetilde h$) the corresponding element of $\widetilde A$ (resp.~$\widetilde{H}$), and we denote by $\{ \widetilde v_t = \widetilde \exp(tJ) \mid t \in \R\}$ the elements of $\widetilde \exp(\mathbb{R}J)$. We also denote by $\widetilde \iota \colon \SU \to \widetilde H$ the isomorphism defined by $\widetilde \iota(u) = \widetilde{\iota(u)}$ for $u \in \SU$. The center of $\widetilde{\mathrm{Sp}}(2,\R)$ is $\{ \widetilde v_t \mid t \in \pi \Z\}$. Since $\R J$ is the center of $\mathfrak{k}$, $\widetilde K = \widetilde \exp(\mathfrak{k})$ is a subgroup isomorphic to $H \times \R$ through the identification of $\widetilde v_t \widetilde h$ with $(h,t)$. The KAK decomposition of $\widetilde{\mathrm{Sp}}(2,\R)$ is summarized in the following lemma.
\begin{lemma} Every element $g \in \widetilde{\mathrm{Sp}}(2,\R)$ can be written as $g=\widetilde h_1 \widetilde v_t \widetilde D(\beta,\gamma) \widetilde v_s \widetilde h_2$ with $\widetilde{h}_1,\widetilde{h}_2 \in \widetilde{H}$, $s,t \in \R$ and $\beta \geq \gamma \geq 0$. Moreover, $s$ can be taken in $[0,\pi)$ (or any other interval of size $\pi$).
\end{lemma}
To continue the description of $\widetilde{\mathrm{Sp}}(2,\R)$, we recall that a quasi-morphism is a map $\Phi: \widetilde{\mathrm{Sp}}(2,\R) \rightarrow \R$ such that $\Phi(g_1 g_2) - \Phi(g_1) - \Phi(g_2)$ is bounded on $\widetilde{\mathrm{Sp}}(2,\R) \times \widetilde{\mathrm{Sp}}(2,\R)$. A quasi-morphism is called trivial if it is bounded (it corresponds to $0$ in the bounded cohomology). Quasi-morphisms on $\widetilde{\mathrm{Sp}}(2,\R)$ (and more generally on the universal cover of $\mathrm{Sp}(n,\R)$ for $n \geq 1$) were studied by Barge and Ghys \cite{BargeGhys}. There are several ways to construct nontrivial quasi-morphisms, but Barge and Ghys proved that all these coincide: there is essentially one quasi-morphism on $\widetilde{\mathrm{Sp}}(2,\R)$, in the sense that the space of quasi-morphisms modulo the bounded quasi-morphisms is one-dimensional. Crucial for our analysis (see the proof of Proposition \ref{prop=strong_T_univcover}) is the fact that this quasi-morphism is bounded on $\widetilde A$ and grows linearly on $\{\widetilde v_t \mid t \in \R\}$. More precisely, we use the following result.
\begin{lemma}\label{lemma=phi_quasimorphism} There is a continuous map $\Phi: \widetilde{\mathrm{Sp}}(2,\R) \to \R$ such that whenever $g \in \widetilde{\mathrm{Sp}}(2,\R)$ has KAK decomposition $g = \widetilde h_1 \widetilde v_t \widetilde D(\beta,\gamma) \widetilde v_s \widetilde h_2$, we have $\Phi(g) = t+s$. Moreover,
\begin{equation}\label{eq=phi_quasimorphism} | \Phi(g_1 g_2) - \Phi(g_1) - \Phi(g_2) | < \frac{\pi}{2}\end{equation}
for all $g_1, g_2 \in \widetilde{\mathrm{Sp}}(2,\R)$.  
\end{lemma}
\begin{proof}
To our knowledge, the function $\Phi$ was introduced (with a different definition) by Guichardet and Wigner \cite{guichardetwigner} and Dupont and Guichardet \cite{dupontguichardet}, and the fact that it is a quasi-morphism is due to Dupont \cite{dupont}. This was reproved in \cite{rawnsley}, in which the function $\Phi$ is the essential ingredient for the explicit realization of $\widetilde{\mathrm{Sp}}(2,\R)$. We will now briefly recall this realization.

Let $c: \mathrm{Sp}(2,\R) \rightarrow S^1$ be the continuous function given by $c(g) = \frac{\det(A+D + i B - iC)}{|\det(A+D + i B - iC)|}$ for $g= \begin{pmatrix} A & C \\ B & D\end{pmatrix}$. Note that $\iota(A+D+iB-iC) = g - JgJ = g + (g^T)^{-1}$. Let the $2$-cocycle $\eta$ be defined as the smooth function $\eta : \mathrm{Sp}(2,\R) \times \mathrm{Sp}(2,\R) \rightarrow \R$ satisfying $c(g_1 g_2) = c(g_1) c(g_2) e^{i\eta(g_1,g_2)}$ and $\eta(1,1)=0$. In this way, we can realize $\widetilde{\mathrm{Sp}}(2,\R)$ as $\{ (g,t) \in \mathrm{Sp}(2,\R) \times \R \mid e^{it} = c(g)\}$ with the group multiplication $(g_1,t_1) (g_2,t_2) = (g_1 g_2,t_1+t_2 + \eta(g_1,g_2))$. We define $\Phi(g,t) = \frac t 2$. By \cite[Lemma 14]{rawnsley}, $\eta$ satisfies $|\eta(g_1,g_2)|<\pi$, which is exactly \eqref{eq=phi_quasimorphism}. To prove the lemma, it remains to show that an element of the form $\widetilde h_1 \widetilde v_t \widetilde D(\beta,\gamma) \widetilde v_s \widetilde h_2$ in $\widetilde{\mathrm{Sp}}(2,\R)$ corresponds to $(h_1 v_t D(\beta,\gamma) v_s h_2,2t+2s)$. We first claim that $c(g) = e^{2i(t+s)}$ for $g=h_1 v_t D(\beta,\gamma) v_s h_2$. Indeed, writing $h_1 v_t = \iota(k_1)$ and $v_s h_2 = \iota(k_2)$ for $k_i \in \mathrm{U}(2)$, we have $(g^T)^{-1} = \iota(k_1) D(-\beta,-\gamma)  \iota(k_2)$, and, hence, $g - JgJ = \iota( k_1 \textrm{diag}(2 \cosh \beta,2\cosh \gamma) k_2)$, from which the equality $c(g) = \det k_1 \det k_2 = e^{2i(t+s)}$ follows. Hence, the expressions $\widetilde h_1 \widetilde v_t \widetilde D(\beta,\gamma) \widetilde v_s \widetilde h_2$ and $(h_1 v_t D(\beta,\gamma) v_s h_2,2t+2s)$ correspond to elements of $\widetilde{\mathrm{Sp}}(2,\R)$ that have the same image in $\mathrm{Sp}(2,\R)$, and both sides depend continuously on $h_1$, $h_2$, $s$, $t$, $\beta$ and $\gamma$. Since they coincide for the identity, we have equality everywhere. 
\end{proof}

\subsection{Banach spaces} \label{subsec=bs}
Let $(g_i)_{i \in \N}$ be a sequence of independent complex Gaussian $\mathcal N(0,1)$ random variables defined on some probability space $(\Omega,\mathbb P)$.

\begin{dfn}\label{def=type_cotype}
A Banach space $X$ has type $p \geq 1$ if there is a constant $T$ such that for all $n \in \mathbb{N}$ and $x_1,\dots,x_n \in X$, we have $\| \sum_i g_i x_i \|_{L^2(\Omega;X)} \leq T \left(\sum_i \|x_i\|^p\right)^{1/p}$. The best $T$ is denoted by $T_p(X)$. A Banach space $X$ has cotype $q \leq \infty$ if there is a constant $C$ such that for all $n \in \mathbb{N}$ and all $x_1,\dots,x_n \in X$, we have $\left(\sum_i \|x_i\|^q\right)^{1/q} \leq C \| \sum_i g_i x_i\|_{L^2(\Omega;X)}$. The best $C$ is denoted by $C_q(X)$. 
\end{dfn}
Hilbert spaces have type $2$ and cotype $2$, and by a theorem of Kwapie{\'n} this property characterizes the Banach spaces that are isomorphic to a Hilbert space. Superreflexive spaces have nontrivial type (i.e., type $>1$), but the converse is not true. Indeed, there are spaces of nontrivial type that are not even reflexive. Also, for every $q>2$, there are Banach spaces that are not reflexive but have type $2$ and cotype $q$ \cite{pisierxu}. For more details, including equivalent definitions of type and cotype, we refer to \cite{maurey}.

Given a Banach space $X$, the following operations give new Banach spaces: (a) taking a subspace of $X$, (b) taking a quotient of $X$, (c) taking an ultrapower of $X$, and (d) taking the complex interpolation space $[X_0,X_1]_\theta$ for some compatible couple $(X_0,X_1)$ with $X_1$ isomorphic to $X$ and $\theta \in (0,1)$.
\begin{dfn}
For $r>2$, let $\mathcal{E}_r$ be the smallest class of Banach spaces that is closed under the operations (a)--(d) and that contains all spaces with type $p$ and cotype $q$ satisfying $\frac{1}{p}-\frac{1}{q} < \frac{1}{r}$.
\end{dfn}
For arbitrary $r>2$, the class $\mathcal{E}_r$ contains non-reflexive spaces. Also, every Banach space in $\mathcal{E}_r$ has nontrivial type. We do not know if $\mathcal{E}_r$ actually depends on $r$. Possibly, it contains all Banach spaces of nontrivial type. For more details, we refer to \cite{delasalle1}.

\subsection{Representations}
A linear representation $\pi:G \rightarrow B(X)$ of a locally compact group $G$ on a Banach space $X$ is said to be continuous if it is continuous with respect to the strong operator topology on $B(X)$. If $m$ is a compactly supported signed Borel measure on $G$ and $\pi:G \rightarrow B(X)$ is a continuous representation, we will denote by $\pi(m)$ the operator on $X$ defined by $\pi(m) \xi = \int \pi(g)\xi dm(g)$ (Bochner integral) for all $\xi \in X$.

Recall that the contragredient representation ${}^T\pi$ of a representation $\pi$ of $G$ on $X$ is the representation of $G$ on $X^*$ given by $g \mapsto \pi(g^{-1})^*$. It might not be continuous, even if $\pi$ is.

Recall that for a representation $\pi:K \rightarrow B(X)$, where $K$ is compact, a vector $\xi \in X$ is of finite $K$-type if $\mathrm{span}(\pi(K) \xi)$ is finite-dimensional. If $V$ is a unitary irreducible representation of $K$, a vector $\xi \in X$ is called of $K$-type $V$ if for every $\eta \in X^*$, the coefficient $k \mapsto \langle \pi(k) \xi,\eta\rangle$ belongs to the space of coefficients of $V$. We denote by $X_V$ the vector space consisting of vectors of $K$-type $V$. By an easy application of the Hahn-Banach Theorem, a vector of $K$-type $V$ is of finite $K$-type.

Let $X_1 \otimes_{\epsilon} X_2$ denote the injective tensor product of $X_1$ and $X_2$. We recall the following analogue of the Peter-Weyl Theorem (see \cite[Theorem 2.5]{delasalle1}).
\begin{thm}[Peter-Weyl Theorem] \label{thm=peterWeyl} Let $K$ be a compact group, and let $\pi:K \rightarrow B(X)$ be a continuous isometric representation.
\begin{enumerate}
\item For every unitary irreducible representation $V$ of $K$, the space $X_V$ is complemented by a $K$-equivariant projection.
\item \label{item2_peterweyl} There is a Banach space $Y$ and a $K$-equivariant isomorphism $u: X_V \rightarrow Y \otimes_{\epsilon} V$ ($K$ acting trivially on $Y$).
\item The vector space spanned by the subspaces $X_V$, where $V$ goes through the equivalence classes of unitary irreducible representations of $K$, is the space $X_{\mathrm{finite}}$ consisting of vectors of finite $K$-type, and it is dense in $X$.
\item The vector space spanned by the subspaces $X^*_V$, where $V$ goes through the equivalence classes of unitary irreducible representations of $K$, is the space $X^*_{\mathrm{finite}}$, and it is weak-* dense in $X^*$.
\end{enumerate}
\end{thm}

\section{Harmonic analysis on $\SU$}\label{section=harmonic_analysis_SU2}
In this section, we consider the group $\SU$ and we discuss certain explicit estimates of the norm of certain operators on $L^2(\SU)$ that we need in Sections \ref{section=sp2} and \ref{section=sp2cov}. This will be important for our approach, as most of the analysis on the representations of $\mathrm{Sp}(2,\mathbb{R})$ and $\widetilde{\mathrm{Sp}}(2,\mathbb{R})$ will be done on the level of their compact subgroups $H$ and $\widetilde H$, which are isomorphic to $\SU$ (see Sections \ref{subsec:psp2} and \ref{subsec:psp2cov}).

The elements of $\SU$ are of the form $\begin{pmatrix} \alpha & \beta \\ -\overline \beta & \overline \alpha\end{pmatrix}$ for $\alpha,\beta \in \C$ satisfying $|\alpha|^2+|\beta|^2 = 1$. For $\theta \in \mathbb{R}$, let
\begin{equation} \nonumber
r_\theta = \begin{pmatrix}\cos \theta & -\sin \theta\\ \sin \theta & \cos \theta\end{pmatrix}, \qquad d_\theta = \begin{pmatrix} e^{i\theta} & 0 \\ 0 & e^{-i\theta}\end{pmatrix}, \qquad u_\theta = \frac{1}{\sqrt 2} \begin{pmatrix} e^{i\theta} & -1 \\ 1 & e^{-i\theta} \end{pmatrix}.
\end{equation}
The set $\{r_\theta\mid \theta \in \mathbb{R}\}$ forms a subgroup of $\mathrm{SU}(2)$ isomorphic to $\SO$, and the set $\{d_\theta\mid \theta \in [0,\frac{\pi}{2}]\}$ forms a family of representatives of the double cosets $\SO\backslash \SU/\SO$. The set $\{d_\theta \mid \theta \in \mathbb{R}\}$ is a subgroup of $\SU$ isomorphic to $\U(1)$. If we denote by $\SU//\U(1)$ the quotient space of $\mathrm{SU}(2)$ with respect to the equivalence relation coming from the conjugation action of $\mathrm{U}(1)$ on $\mathrm{SU}(2)$, then $\{u_\theta\mid \theta \in [0,2\pi)\}$ forms a subfamily of the set of representatives of the corresponding equivalence classes. Let $P = \int_\SO \lambda(r) dr \in B(L^2(\SU))$ denote the orthogonal projection onto the space of $\SO$-invariant functions on $\SU$. For $\theta \in \R$, let $T_\theta =P \lambda(d_\theta) P \in B(L^2(\SU))$, and let $S_\theta \in B(L^2(\SU))$ be the operator given by
\[ S_\theta \xi= \frac{1}{2\pi}\int_0^{2\pi} \lambda(d_\varphi u_\theta d_{-\varphi})\xi d\varphi,\] where $\lambda$ denotes the left-regular representation of $\SU$.
\begin{lemma}\label{lemma=estimate_T_theta_in_schatten_classes}
 For $p>4$, there is a constant $C_p$ such that for $\theta \in [\frac{\pi}{6},\frac{\pi}{3}]$, 
 \[ \| T_{\theta} - T_{\frac{\pi}{4}} \|_{S^p} \leq C_p \left|\theta - \frac{\pi}{4}\right|^{\frac{1}{2} - \frac{2}{p}}.\]
\end{lemma}
\begin{proof}
This is \cite[Lemma 5.3]{lafforguedelasalle} in disguise. To avoid confusion, for $\delta \in [-1,1]$, let $\widetilde T_\delta \in B(L^2(\mathbb S^2))$ denote the operators appearing that lemma. We recall that for a function $f$ on $\mathbb S^2$ and $x \in \mathbb S^2$, $\widetilde T_\delta f(x)$ is the average of $f$ on $\{y \in \mathbb S^2 \mid \langle x,y\rangle = \delta\}$. We claim that there is a unitary $U \colon L^2(\mathbb S^2) \to \mathrm{Ran}(P)$ such that $U^* T_\theta U = \widetilde T_{\cos(2\theta)}$ for all $\theta$. This implies that $\| T_\theta - T_{\frac{\pi}{4}}\|_{S^p} = \| \widetilde T_{\cos 2 \theta} - \widetilde T_0\|_{S^p}$, and the estimate we have to prove follows immediately from the first inequality of \cite[Lemma 5.3]{lafforguedelasalle}.

To construct the unitary $U$, we recall the classical isomorphism between $\mathrm{SO}(3)$ and $\SU/\{-1,1\}$ obtained from the adjoint action of $\SU$ on its Lie algebra $\mathfrak{su}_2$. Explicitly, this isomorphism is given by
\begin{equation}\label{eq=isom_SO3_SU2} \begin{pmatrix} \alpha & \beta \\ -\overline \beta & \overline \alpha\end{pmatrix} \mapsto \begin{pmatrix} \mathrm{Re}(\alpha^2+ \beta^2) & \mathrm{Im}(\beta^2 - \alpha^2) & 2 \mathrm{Im}(\alpha \beta) \\ \mathrm{Im}(\alpha^2+\beta^2)& \mathrm{Re}(\alpha^2 - \beta^2)& -2 \mathrm{Re}(\alpha \beta)\\ 2 \mathrm{Im}(\alpha \overline \beta) & 2\mathrm{Re}(\alpha \overline \beta)& |\alpha|^2 - |\beta|^2 \end{pmatrix},\end{equation}
where $\alpha,\beta \in \C$ satisfy $|\alpha|^2 + |\beta|^2 = 1$.
This isomorphism also induces an isomorphism $\{ r_\theta \mid \theta \in \R\}/\{-1,1\} \cong \begin{pmatrix} 1 & 0 \\ 0 & \mathrm{SO}(2)\end{pmatrix}$, from which we get isomorphisms $L^2(\{ r_\theta \mid \theta \in \R\} \backslash \SU) \cong L^2 \left( \begin{pmatrix} 1 & 0 \\ 0 & \mathrm{SO}(2)\end{pmatrix} \backslash \mathrm{SO}(3) \right) \cong L^2(\mathbb S^2)$. This gives our unitary $U^*$, because the image of $P$ is $L^2(\{ r_\theta \mid \theta \in \R\} \backslash \SU)$. The idenfication $U^* T_\theta U = \widetilde T_{\cos(2\theta)}$ is straightforward by \eqref{eq=isom_SO3_SU2}.
\end{proof}
The following lemma is proved in the appendix.
\begin{lemma}\label{lemma=estimate_S_theta_in_schatten_classes}
 For $p>10$, there is a constant $C'_p$ such that for all $\theta_1,\theta_2 \in \mathbb{R}$,
 \[ \| S_{\theta_1} - S_{\theta_2} \|_{S^p} \leq C'_p |\theta_1 - \theta_2|^{\frac{1}{4} - \frac{5}{2p}}.\]
\end{lemma}
The estimate for the $T_{\theta}$'s is implicit in \cite{haagerupdelaat1} for $p=\infty$ and in \cite{delaat1} for $p>4$; an analogous estimate for the $S_{\theta}$'s is implicit in \cite{haagerupdelaat2} for $p > 12$.

For the non-invariant coefficients of representations (see Sections \ref{section=sp2} and \ref{section=sp2cov}), we need the following two propositions, which rely on \cite[Proposition 2.8]{delasalle1}.
\begin{prop}\label{prop:harmonic_ana_SO2}
For every integer $m$, there is a contant $C(m)$ (with $C(0)=1$) such that for every Banach space $X$, every isometric representation $\pi:\mathrm{SU}(2) \rightarrow B(X)$ and all unit vectors $\xi \in X$, $\eta \in X^*$ for which the coefficient $c(h) = \langle \pi(h) \xi,\eta\rangle$ satisfies $c\left( r_\theta  h r_{\theta'} \right) = e^{i m \theta} c(h)$ for all $h \in \mathrm{SU}(2)$ and $\theta,\theta' \in \mathbb{R}$, we have
\[ |c(d_\theta) - c(d_{\frac{\pi}{4}})| \leq C(m) \|(T_\theta - T_{\frac{\pi}{4}}) \otimes \mathrm{id}_X\|_{B(L^2(\SU;X))}\]
for all $\theta \in [0,\frac{\pi}{4}]$.
\end{prop}
\begin{proof}
If $m$ is odd, this is completely obvious, since the hypothesis $c\left( r_\theta  h r_{\theta'} \right) = e^{i m \theta} c(h)$ and the fact that $r_\pi = -1$ belongs to the center of $\SU$ imply (taking $\theta = \theta'=\pi$) that $c$ is identically $0$. 

If $m=0$, this is the observation of \cite[Proposition 2.7]{delasalle1}. Therefore, let us assume $m\neq 0$ is even.

A simple computation shows that if $r_\theta d_{\frac{\pi}{4}} r_{\theta'} = d_{\frac{\pi}{4}}$, then $r_\theta = r_{\theta'} = \pm 1$. In other words, the action of the subgroup $\mathrm{SO}(2) \times \mathrm{SO}(2) \subset \mathrm{SU}(2) \times \mathrm{SU}(2)$ by left-right multiplication on $\SU$ satisfies the condition of \cite[Proposition 2.8]{delasalle1} for $x_0 = d_{\frac{\pi}{4}}$ and the character $\chi(r_\theta,r_{\theta'}) = e^{im\theta}$. Hence, we get a constant $C(m)$ such that 
\[ |c(d_\theta) - c(d_{\frac{\pi}{4}})| \leq C(m) \left( \left|\theta - \frac \pi 4\right|+\|(T_\theta - T_{\frac{\pi}{4}}) \otimes \mathrm{id}_X\|_{B(L^2(\SU;X))}\right).\]
We claim that 
\begin{equation}\label{eq=lower_inequality_T_theta}\|(T_\theta - T_{\frac{\pi}{4}})\|_{B(L^2(\SU))} \geq |\cos(2\theta)|.\end{equation} This implies that there is a $c \in \R$ such that $|\theta - \frac \pi 4| \leq c \|(T_\theta - T_{\frac{\pi}{4}})\|_{B(L^2(\SU))}$ for all $\theta \in [0,\frac{\pi}{4}]$, which concludes the proof. The inequality \eqref{eq=lower_inequality_T_theta} can be obtained from the analysis in the proofs of Lemma \ref{lemma=estimate_T_theta_in_schatten_classes} and \cite[Lemma 5.3]{lafforguedelasalle}, or can be directly obtained by doing the computation $T_\theta f = \cos(2\theta) f$, where the function $f \in L^2(\SU)$ is given by \[ f \begin{pmatrix} \alpha & \beta \\ -\overline \beta & \overline \alpha\end{pmatrix} = \mathrm{Re}(\alpha^2+\beta^2).\]
\end{proof}
\begin{prop}\label{prop:harmonic_ana_U1}
For every integer $m$, there is a contant $C^{\prime}(m)$ (with $C^{\prime}(0)=1$) such that for every Banach space $X$, every isometric representation $\pi:\mathrm{SU}(2) \rightarrow B(X)$ and all unit vectors $\xi \in X$, $\eta \in X^*$ for which the coefficient $c(h) = \langle \pi(h) \xi,\eta\rangle$ satisfies $c\left(d_\theta h d_{-\theta} \right) = e^{i m \theta} c(h)$ for all $h \in \mathrm{SU}(2)$ and $\theta \in \mathbb{R}$, we have
\[ |c(u_{\theta}) - c(u_{\frac{\pi}{2}})| \leq C'(m) \|(S_{\theta} - S_{\frac \pi 2}) \otimes \mathrm{id}_X\|_{B(L^2(\SU;X))}\]
for all $\theta \in \mathbb{R}$.
\end{prop}
\begin{proof} Again, the case $m$ odd is obvious because $d_\pi = -1$ belongs to the center of $\SU$. The case $m=0$ is \cite[Proposition 2.7]{delasalle1}. For $m \neq 0$ even, consider the subgroup $\{(d_\varphi,d_\varphi) \mid \varphi \in \R/2\pi \Z\} \subset \SU \times \SU$. As above, we can apply \cite[Proposition 2.8]{delasalle1} for $x_0 = u_{\frac{\pi}{2}}$, which yields the required inequality, provided that $\|(S_{\theta} - S_{\frac \pi 2})\|_{B(L^2(\SU))} \geq 2^{-\frac{1}{2}}|e^{i\theta} - e^{i\frac{\pi}{2}}|$. This inequality is \eqref{eq=lower_inequality_S_theta}.
\end{proof}

\section{Decay of representations of $\mathrm{Sp}(2,\mathbb{R})$ on certain Banach spaces}\label{section=sp2}
In this section, let $G=\mathrm{Sp}(2,\mathbb{R})$ (see Section \ref{subsec:psp2}). We prove explicit decay estimates for matrix coefficients of representations of $G$ with small exponential growth on certain Banach spaces.
\subsection{Statement of the result}
From now on, we assume that $X$ is a Banach space for which there exist $s_1,s_2>0$ and $C_1,C_2 \in\R$ such that for $\theta \in \R$,
\begin{align}
 \label{eq:assumption_on_X_SO2} \|\left(T_\theta - T_{\frac{\pi}{4}}\right) \otimes Id_X \|_{B(L^2(\SU;X))} &\leq C_1 \left|\theta - \frac{\pi}{4}\right|^{s_1}, \\
 \label{eq:assumption_on_X_U1} \|\left(S_\theta - S_{\frac{\pi}{2}}\right) \otimes Id_X \|_{B(L^2(\SU;X))} &\leq C_2 \left|\theta - \frac{\pi}{2}\right|^{s_2}.
\end{align}
\begin{rem}\label{rem=s1s2leq1} By \eqref{eq=lower_inequality_T_theta} and \eqref{eq=lower_inequality_S_theta}, $s_1\leq 1$ and $s_2 \leq 1$ if $X$ is nonzero.
\end{rem}
For $s_1,s_2>0$, let $s_-(s_1,s_2)$ be the smallest root of the polynomial $P(x) = x^2 - (2s_1+s_2)x + s_1 s_2$. Explicitly, $s_-(s_1,s_2) = s_1 + \frac{s_2}{2} - \sqrt{s_1^2+\frac{s_2^2}{4}}$. Note that $0<s_-(s_1,s_2)<s_1$ and $0<s_-(s_1,s_2)<s_2$. For fixed values of $s$ satisfying $0 < s < s_-(s_1,s_2)$, we will consider representations $\pi:G \rightarrow B(X)$ for which there exists an $L > 0$ such that for $\beta \geq \gamma \geq 0$,
\begin{equation} \label{eq:assumption_on_pi}
  \| \pi(D(\beta,\gamma)) \|_{B(X)} \leq L e^{s\beta}.
\end{equation}

The following result is the main theorem of this section.
\begin{thm} \label{thm:decay_of_coefficients_Sp2}
Let $X$ be a Banach space satisfying \eqref{eq:assumption_on_X_SO2} and \eqref{eq:assumption_on_X_U1}, and let $s<s_{-}(s_1,s_2)$. There exists $\varepsilon \in C_0(G)$ such that the following holds: for a representation $\pi:G \rightarrow B(X)$ satisfying \eqref{eq:assumption_on_pi}, a unitary irreducible representation $V$ of $\SU$, and $m \in \Z$, there exists a constant $C$ such that for every $\xi \in X^H$ and $\eta \in (X^*)_V$ unit vectors,
\begin{itemize}
 \item if $V$ is the trivial representation and $m=0$, there exists a $c_{\xi,\eta} \in \C$ such that
\begin{equation} \label{eq:equivariant_coefficients_Sp2}
 \bigg\vert\int_{\R/\pi\Z} \langle \pi(v_t g v_{-t}) \xi,\eta\rangle \frac{d t}{\pi} - c_{\xi,\eta}\bigg\vert \leq C \varepsilon(g);
\end{equation}
\item otherwise, 
\begin{equation} \label{eq:nonequivariant_coefficients_Sp2}
 \bigg\vert\int_{\R/\pi\Z} e^{-2imt}\langle \pi(v_t g v_{-t}) \xi,\eta\rangle \frac{d t}{\pi}\bigg\vert \leq C \varepsilon(g). 
\end{equation}
\end{itemize}
Moreover there is a $c>0$ such that $\varepsilon( k D(\beta,\gamma) k') \leq e^{-c \beta}$ for all $\beta \geq \gamma \geq 0$ and $k,k' \in K$.
\end{thm}
For a more precise form of $\varepsilon$, see Remark \ref{rem=precise_value_epsilon}.
\begin{rem}
An analogous result holds for the coefficients of $\pi$ with respect to $K$-finite vectors, as for $\mathrm{SL}(3,\mathbb{R})$ and coefficients of representations with respect to $\mathrm{SO}(3)$-finite vectors (see \cite[Proposition 4.3]{delasalle1}). This follows from Theorem \ref{thm:decay_of_coefficients_Sp2}, but could also be proved directly. However, such a result does not extend to the universal covering group of $G$, since $\widetilde{K}$ is not compact. We state the theorem in the way above, since the proof works with almost no change for the universal covering group of $G$.
\end{rem}

The proof of Theorem \ref{thm:decay_of_coefficients_Sp2} relies on a series of technical lemmas, and takes the rest of this section. We start by explaining the general idea of the proof.

\subsection{Outline of the proof} The proof follows the same general strategy as the proof of the fact that $\mathrm{SL}(3,\mathbb{R})$ has (T$^{\mathrm{strong}}_{\mathrm{Hilbert}}$) in \cite{lafforguestrengthenedpropertyt}. It relies heavily on Propositions \ref{prop:harmonic_ana_SO2} and \ref{prop:harmonic_ana_U1}.

In the proof of \eqref{eq:equivariant_coefficients_Sp2}, we set $c(g) = \int_{\R/\pi\Z} \langle \pi(v_t g v_{-t}) \xi,\eta\rangle \frac{d t}{\pi}$. The KAK decomposition of $G$ plays a key role in the study of $c$, as an element $g \in G$ with KAK decomposition $g= h_1v_t D(\beta,\gamma) v_sh_2$ satisfies $c(g) = c(v_{t+s} D(\beta,\gamma))$. This indicates that it will be useful to compute the precise KAK decomposition of certain elements of $G$ (or at least the $A$-part and the value of $t+s$). Now, on the one hand, if $D_1,D_2 \in G$ commute with $\iota(r_\theta)$ for all $\theta$, then $c_{D_1,D_2} \colon \mathrm{SU}(2) \to \mathbb{C},\, h \mapsto c(D_1 \iota(h) D_2)$ is a coefficient of an isometric representation of $\SU$ with respect to vectors of norm dominated by $\|\pi(D_1)\|$ and $\|\pi(D_2)\|$ that satisfies $c_{D_1,D_2}(r_\theta h r_{\theta'}) = c_{D_1,D_2}(h)$ for all $r_\theta,r_{\theta'}$. Proposition \ref{prop:harmonic_ana_SO2} therefore applies and, with \eqref{eq:assumption_on_X_SO2}, gives certain local H\"older continuity estimates for $c$. On the other hand, by considering $c_{D_1,D_2}$ for $D_1,D_2 \in G$ commuting with $\iota(d_\theta)$ for all $\theta$, Proposition \ref{prop:harmonic_ana_U1} and \eqref{eq:assumption_on_X_U1} give other local H\"older continuity estimates for $c$. The idea is to combine such estimates obtained for all suitable choices of $D_1$ and $D_2$ in order to show that $c(g)$ satisfies the Cauchy criterion as $g \to \infty$ and hence has a limit. 

In order to prove \eqref{eq:nonequivariant_coefficients_Sp2}, we proceed as follows. By Theorem \ref{thm=peterWeyl}, we can write $(X^*)_V = Y \otimes V$ for some Banach space $Y$, and we can assume that $\eta = y \otimes v_0$ for some $y \in Y$ and $v_0 \in V$. Define the $K$-equivariant map $q\colon X \to V^*$ by $\langle q(x),v\rangle = \langle x, y \otimes v\rangle$ for all $v \in V$, and consider $c(g) = \int_{\R/\pi\Z} e^{-2imt} q(\pi(v_t g v_{-t}) \xi) \frac{d t}{\pi}$. Then \eqref{eq:nonequivariant_coefficients_Sp2} is equivalent to $\|c(g)\|_{V^*} \leq C \varepsilon(g)$ (for a different constant $C$). Again, $c(g)$ is better understood through the KAK decomposition of $g$, but this time the full KAK decomposition is needed. Then, as for \eqref{eq:equivariant_coefficients_Sp2}, we study $c_{D_1,D_2} \colon \mathrm{SU}(2) \to \mathbb{C},\, h \mapsto c(D_1 \iota(h) D_2)$. If $D_1$ and $D_2$ commute with $\{\iota(r_\theta) \mid \theta \in \R\}$, Proposition \ref{prop:harmonic_ana_SO2} and a decomposition of $V^*$ into characters of the commutative subgroup $\{r_\theta \mid \theta \in \R\}$ give local H\"older continuity estimates for $c$. Similarly, if $D_1$ and $D_2$ commute with $\iota(d_\theta)$ for all $\theta$, we get other estimates. By combining such estimates, one can show that for every $g \in G$, $c(g)$ is simultaneously close to $c(v_t D(\alpha,0))$ and to $c(v_s D(\alpha^{\prime},\alpha^{\prime}))$ for some $\alpha,\alpha',s,t \in \R$. By using that $v_s D(\alpha^{\prime},\alpha^{\prime})$ commutes with every $\iota(r_\theta)$, and similarly by considering the commutation relation of $v_t D(\alpha,0)$ and $\iota(d_\theta)$, we conclude that in order for the vectors $c(v_t D(\alpha,0))$ and $c(v_s D(\alpha^{\prime},\alpha^{\prime}))$ in $V^*$ to be close to each other, they have to be close to $0$. Hence, $c(g)$ has to be close to $0$. 

\subsection{Computations on the KAK decomposition}
An important technical step in the proof sketched above is to obtain precise KAK decompositions of elements of the form $D_1 \iota(h) D_2$. This is the content of the next series of lemmas. Parts of these computations have been obtained in \cite[Lemmas 3.9 and 3.15]{haagerupdelaat1} and \cite[Lemmas 3.18 and 3.23]{haagerupdelaat2}, but here we need more information on the $K$-part. For this, we will use the explicit KAK decomposition of certain elements in $\mathrm{GL}(2,\mathbb{R})$.
\begin{lemma}\label{lemma=polar_decomposition_in_SL2_general}
Let $a,d\geq 0$, and $c \in \R$. Then one can write
 \begin{equation} \nonumber
 \begin{pmatrix} a & -c\\ c & d \end{pmatrix} =  r_{\phi} 
    \begin{pmatrix}
     \lambda & 0 \\ 0 & \mu
    \end{pmatrix} r_{\phi},
 \end{equation}
where $\phi \in [-\frac{\pi}{4},\frac{\pi}{4}]$ and $\lambda,\mu \geq 0$ are characterized by
\begin{equation} \label{eq=polar_SL2}
\begin{split}
  \lambda \mu &= ad+c^2,\\
  \lambda - \mu &= a-d,\\
  \tan(2\phi) &= \frac{2c}{a+d},
\end{split}
\end{equation}
with the convention that, in case $a=d=0$, we have $\phi=\frac{\pi}{4}$ if $c\geq 0$, and $\phi=-\frac{\pi}{4}$ if $c< 0$.
\end{lemma}
This lemma follows by computations that we leave to the reader. The following lemma gives a special case of this result.
\begin{lemma}\label{lemma=polar_decomposition_in_SL2}
 Let $\alpha>0$ and $\theta \in [-\frac{\pi}{2},\frac{\pi}{2}]$. We can write
 \begin{equation} \nonumber
 \begin{pmatrix} 
     e^\alpha & 0 \\ 0 & e^{-\alpha} 
    \end{pmatrix} r_\theta \begin{pmatrix}
     e^{\alpha} & 0 \\ 0 & e^{-\alpha} 
    \end{pmatrix} = r_{\phi} 
    \begin{pmatrix}
     e^{\beta} & 0 \\ 0 & e^{-\beta}
    \end{pmatrix} r_{\phi},
 \end{equation}
where $\beta \geq 0$ and $\phi \in [-\frac{\pi}{4},\frac{\pi}{4}]$ are characterized by
\begin{align}
\label{eq=polardecSL2_beta} \sinh \beta &= \sinh(2\alpha) \cos \theta,\\
\label{eq=polardecSL2_phi} \tan(2\phi) &= \frac{\tan \theta}{\cosh(2\alpha)},\end{align}
with the convention that $\tan(\frac{\pi}{2})=\infty$ and $\tan(-\frac{\pi}{2}) =-\infty$.
\end{lemma}
We can now compute certain KAK decompositions in $G$. For $\theta,\theta' \in \R$, we introduce the following element of $K$:
\begin{equation} \nonumber
w_{\theta,\theta'} = \iota\begin{pmatrix} e^{i\theta} & 0 \\ 0 & e^{i \theta'} \end{pmatrix} = v_{\frac{\theta+\theta'}{2}} \iota(d_{\frac{\theta-\theta'}{2}}).
\end{equation}
\begin{prop}\label{prop=structural_SO2}
For all $\alpha > 0$ and $\theta \in [0,\frac{\pi}{2}]$, we have
\begin{align*}D(\alpha, \alpha)w_{\theta,\frac{\pi}{2}-\theta} D(\alpha, \alpha) &= w_{\phi,\phi'} D(\beta, \gamma)w_{\phi,\phi'}, \\ D(\alpha, \alpha)w_{-\theta,\theta-\frac{\pi}{2}} D(\alpha, \alpha) &= w_{-\phi,-\phi'} D(\beta, \gamma)w_{-\phi,-\phi'},\end{align*}
where $\beta, \gamma \geq 0$ and $\phi, \phi' \in [0, \frac{\pi}{4}]$ satisfy
\begin{align}
\label{eq=structural_SO2_D} \sinh\beta &= \sinh(2\alpha) \cos \theta, \quad & \sinh\gamma &= \sinh(2\alpha) \sin \theta,\\
\label{eq=structural_SO2_phi} \tan(2\phi) &= \frac{\tan \theta}{\cosh(2\alpha)}, \quad & \tan(2\phi^{\prime}) &= \frac{1}{\tan \theta \cosh(2\alpha)}.
\end{align}
\end{prop}
\begin{proof} Let $e_1,e_2,e_3,e_4$ be the standard basis of $\R^4$. Then each matrix in this proposition leaves the planes $\R e_1 + \R e_3$ and $\R e_2 + \R e_4$ invariant. The result follows by restricting to these subspaces and applying Lemma \ref{lemma=polar_decomposition_in_SL2}.
\end{proof}
\begin{prop}\label{prop=structural_U1}
For all $\alpha>0$ and $\theta \in [-\frac{\pi}{2},\frac{\pi}{2}]$, we have
\[D(\alpha,0) \iota(u_\theta) D(\alpha,0) = \iota(r_\phi)w_{\omega_1,\omega_2} D(\beta,\gamma) w_{\omega_1,\omega_2} \iota(r_\phi),\]
where $\beta \geq \gamma \geq 0$ and $\phi ,\omega_1,\omega_2 \in [-\frac{\pi}{4},\frac{\pi}{4}]$ satisfy
\begin{align}
\label{eq=prod_beta_gamma} \sinh \beta \sinh \gamma &= \frac 1 2 \sinh^2(\alpha),\\
\label{eq=diff_beta_gamma} \sinh \beta - \sinh \gamma &= \frac{1}{\sqrt{2}}\sinh(2\alpha) \cos \theta,\\
\label{eq=structU1_phi}\tan(2\phi) &= \frac{1}{\cosh \alpha \cos \theta},\\
\label{eq=structU1omega}\sin(2\omega_1) \cosh \beta &= \frac{1}{\sqrt{2}}\sin \theta = - \sin(2\omega_2) \cosh\gamma.
\end{align}
\end{prop}
\begin{proof}
This is just a computation with $4 \times 4$ matrices. We still give a detailed proof, as the computations are slightly more involved than for Proposition \ref{prop=structural_SO2}. We assume that $\theta \in (-\frac{\pi}{2},\frac{\pi}{2})$. The case $|\theta|= \frac{\pi}{2}$ will follow by continuity. Let $X = D(\alpha,0) \iota(u_\theta) D(\alpha,0)$. It follows from the KAK decomposition of $G$ that there exist $\beta \geq \gamma \geq 0$ and $k,k' \in K$ such that $X = k D(\beta,\gamma) k'$. In general, the element $k$ is not unique, but it is determined uniquely up to multiplication on the right by a matrix in $\mathcal D:=\{\iota(\diag(\epsilon_1,\epsilon_2)) \mid \epsilon_i \in \{-1,1\}\}$. We prove that the proposition holds, perhaps with $k$ replaced by a matrix in $k \mathcal D$. Computing $\frac{1}{2}(X - (X^{-1})^T)$ gives
\[  \frac 1 {\sqrt 2} \begin{pmatrix} Y & 0  \\ 0 & -Y \end{pmatrix} = k \diag(\sh \beta,\sh \gamma,-\sh \beta,-\sh \gamma) k', \]
where $Y= \begin{pmatrix} \sh(2\alpha) \cos \theta & -\sh \alpha
  \\ \sh \alpha & 0\\\end{pmatrix}$. By Lemma \ref{lemma=polar_decomposition_in_SL2_general}, we have $Y= r_\phi \diag(\lambda,\mu) r_\phi$, with $\lambda \geq \mu \geq 0$ and $\phi \in (-\frac{\pi}{4},\frac{\pi}{4})$ satisfying  $\lambda \mu = \sinh^2 \alpha$, and $\lambda - \mu = \sinh(2\alpha) \cos \theta$, and $\tan(2\phi) = 2 \sinh \alpha (\sinh(2\alpha) \cos \theta)^{-1}=(\cosh \alpha \cos \theta)^{-1}$. For $k_0:=\iota(r_\phi) \in K$, we obtain
\begin{equation} \label{eq=relation_k0_k} \frac{1}{\sqrt{2}} k_0 \diag(\lambda,\mu,-\lambda,-\mu) k_0 = k \diag(\sh \beta,\sh \gamma,-\sh \beta,-\sh \gamma) k'.\end{equation}
This implies that $\sinh \beta = \frac{\lambda}{\sqrt{2}}$ and $\sinh \gamma=\frac{\mu}{\sqrt{2}}$, so that \eqref{eq=prod_beta_gamma} and \eqref{eq=diff_beta_gamma} hold. 
If $\theta=0$, the matrix $X$ is block-diagonal, in which case we can take $k=k'=k_0$. For other values of $\theta$, we can still assume that $k$ and $k'$ depend continuously on $\theta$, which determines $k$ and $k'$ uniquely. From \eqref{eq=relation_k0_k}, we get that $k_0^{-1} k $ commutes with $\diag(\sh^2 \beta,\sh^2 \gamma,\sh^2 \beta,\sh^2 \gamma)$, which implies (since $\beta > \gamma > 0$ by our assumption that $\cos \theta \neq 0$) that $\iota^{-1}(k_0^{-1} k)$ is a diagonal matrix, i.e., $k = \iota(r_\phi \diag(e^{i\omega_1},e^{i\omega_2})) = \iota(r_\phi)w_{\omega_1,\omega_2}$ for $\omega_1,\omega_2 \in\R$. Substituting this in \eqref{eq=relation_k0_k}, we obtain $k' = w_{\omega_1,\omega_2}\iota(r_\phi)$. By our previous remark, we can assume that $\omega_1,\omega_2$ depend continuously on $\theta$ and take $\omega_1=\omega_2=0$ if $\theta=0$.

Consider the lower-left $2 \times 2$-submatrix in the equality $X = k D(\beta,\gamma) k'$:
\[ \frac{\sin\theta}{\sqrt 2}
\begin{pmatrix} 1 & 0\\0&-1\end{pmatrix}
 = r_\phi 
\begin{pmatrix} \sin(2 \omega_1) \cosh \beta & 0 \\ 
0 & \sin(2\omega_2) \cosh\gamma \end{pmatrix}
 r_\phi.\]
This equality implies that $\sin(2\omega_1) \cosh \beta = \frac{1}{\sqrt{2}}\sin \theta = -\sin(2\omega_2) \cosh\gamma$. In particular, $|\sin 2\omega_j|<1$, and hence $|\omega_j|< \frac{\pi}{4}$ by continuity, since $\omega_1=\omega_2=0$ if $\theta=0$.
\end{proof}

\subsection{Invariant coefficients} \label{subsec=invc}
\begin{notation}
In Sections \ref{subsec=invc} and \ref{subsec=noninvc}, we will use the following notation. For numerical expressions, we write $A \lesssim B$ if there exists a constant $C > 0$ such that $A \leq CB$. The constant $C$ can depend on $X$ (through the constants $C_1,s_1$ and $C_2,s_2$ in \eqref{eq:assumption_on_X_SO2} and \eqref{eq:assumption_on_X_U1}, respectively), $L$, $V$ and $m$.
\end{notation}
We now prove the first part of Theorem \ref{thm:decay_of_coefficients_Sp2}. Let $X$ be a Banach space satisfying \eqref{eq:assumption_on_X_SO2} and \eqref{eq:assumption_on_X_U1}, and let $\pi:G \rightarrow B(X)$ be a continuous representation satisfying \eqref{eq:assumption_on_pi}. Replacing the norm on $X$ by the equivalent norm $ \|x\|' = \int_{K} \|\pi(k) x\| dk$, we may assume that the restriction of $\pi$ to $K$ is isometric. For $H$-invariant unit vectors $\xi \in X$ and $\eta \in X^*$, let
\begin{equation} \nonumber
  c(g) = \int_{\R/\pi\Z} \langle \pi(v_t g v_{-t}) \xi,\eta\rangle \frac{d t}{\pi}.
\end{equation}
It follows that $c(h_1v_tgv_{-t}h_2)=c(g)$ for all $g \in G$, $h_1,h_2 \in H$ and $t \in \mathbb{R}$. Note that $c$ is a coefficient of the representation $1 \otimes \pi$ of $G$ on $L^2(\R/\pi \Z;X)$. Indeed, $c(g) = \langle (1 \otimes \pi)(g) \widetilde \xi,\widetilde \eta \rangle$, where $\widetilde \xi \in L^2(\R/\pi\Z,dt/\pi;X)$ and $\widetilde \eta \in L^2(\R/\pi\Z,dt/\pi;X^*)$ are the vectors $\widetilde \xi(t) = \pi(v_{-t}) \xi$ and $\widetilde \eta(t) = {}^T\pi(v_{-t}) \eta$, which have norm $1$, as $\pi(v_{-t})$ is an isometry. In what follows, we will use that by Fubini's Theorem, \eqref{eq:assumption_on_X_SO2} and \eqref{eq:assumption_on_X_U1} also hold for $X$ replaced by $L^2(\R/\pi\Z;X)$.

With the notation of Propositions \ref{prop=structural_SO2} and \ref{prop=structural_U1}, it follows that for every $t \in \R/\pi\Z$,
\begin{align}
  c(v_{t}D(\alpha,\alpha) w_{-\theta,\theta - \frac{\pi}{2} }D(\alpha,\alpha)) &= c(v_{t+\phi+\phi'}D(\beta,\gamma)), \label{eq=value_of_c_circles_equivariant} \\
  c(v_t D(\alpha,0) \iota(u_\theta) D(\alpha,0)) &= c(v_{t+\omega_1+\omega_2} D(\beta,\gamma)). \label{eq=value_of_c_hyperbolas_equivariant}
\end{align}
\begin{lemma} \label{lemma:eqcir}
Let $c$ be as above, let $\beta \geq \gamma \geq 0$, and let $\alpha$ be the non-negative solution of $\sinh^2(2\alpha)=\sinh^2\beta + \sinh^2\gamma$. Then there exists a $\rho \in \mathbb{R}/2\pi\mathbb{Z}$ such that for all $t \in \mathbb{R}/2\pi\mathbb{Z}$,
\[
  \left|c(v_tD(\beta,\gamma))-c(v_{t+\rho}D(2\alpha,0))\right| \lesssim e^{s\beta - s_1(\beta-\gamma)}.
\]
\end{lemma}
\begin{proof}
  For $t \in \R/2\pi\Z$, let $c_{t,\alpha}:h \mapsto c(v_{t} D(\alpha,\alpha) v_{-\frac{\pi}{4}} \iota(h) D(\alpha,\alpha))$ be defined on $\mathrm{SU}(2)$. Let $\widetilde{\xi}$ and $\widetilde{\eta}$ be as above. For $\xi_{t,\alpha} = (1 \otimes \pi(D(\alpha,\alpha))) \widetilde \xi$ and $\eta_{t,\alpha} = (1 \otimes \pi(v_tD(\alpha,\alpha)v_{-\frac{\pi}{4}}))^* \widetilde \eta$, which are vectors of norm at most $L e^{s\alpha}$ in $L^2(\R/\pi \Z,dt/\pi;X)$ and $L^2(\R/\pi \Z,dt/\pi;X^*)$, respectively, we can write $c_{t,\alpha}=\langle (1\otimes \pi)(\iota(h)) \xi_{t,\alpha},\eta_{t,\alpha}\rangle$. Since $D(\alpha,\alpha)$ and $v_t$ commute with $\iota(r_\theta)$ for all $\theta$, we have $c_{t,\alpha}(r_\theta h r_{\theta'}) = c_{t,\alpha}(h)$ for all $\theta,\theta^{\prime} \in \mathbb{R}$.
  
  For $h=d_{\frac{\pi}{4}-\theta}$ and $h^{\prime}=d_{\frac{\pi}{4}}$, it follows from Proposition \ref{prop:harmonic_ana_SO2} and \eqref{eq:assumption_on_X_SO2} that
\[
  \left|c(v_tD(\alpha,\alpha)w_{-\theta,\theta-\frac{\pi}{2}}D(\alpha,\alpha))-c(v_tD(\alpha,\alpha)w_{0,-\frac{\pi}{2}}D(\alpha,\alpha))\right| \lesssim e^{2s\alpha}|\theta|^{s_1},
\]
provided that $\theta \in [0,\frac{\pi}{4}]$. Take $\theta = \arctan(\sinh \gamma/\sinh \beta)$. Then $|\theta| \leq e^{\gamma - \beta}$ and $2\alpha \leq \beta + 1$, so the right-hand side is $\lesssim e^{s \beta -s_1(\beta- \gamma)}$. By \eqref{eq=value_of_c_circles_equivariant}, we have $c(v_{t}D(\alpha,\alpha)w_{-\theta,\theta-\frac{\pi}{2}}D(\alpha,\alpha)) = c(v_{t+\rho_1} D(\beta,\gamma))$ and $c(v_{t}D(\alpha,\alpha)w_{0,-\frac{\pi}{2}}D(\alpha,\alpha)) = c(v_{t+\rho_2} D(2\alpha,0))$ for some $\rho_1,\rho_2 \in \R/2\pi \Z$, Hence, 
\[
  \left| c( v_{t} D(\beta,\gamma)) - c(v_{t+\rho} D(2\alpha,0))\right| \lesssim e^{\beta s - s_1(\beta-\gamma)}
\]
for $\rho=\rho_2-\rho_1$.
\end{proof}
\begin{lemma} \label{lemma:dependency_on_t}
Let $c$ be as above, and let $\alpha > 0$. Then for all $t \in \mathbb{R}/2\pi\mathbb{Z}$ and $\tau \in [0,\frac{\pi}{4}]$,
\[
  |c(v_tD(2\alpha,0))-c(v_{t+\tau}D(2\alpha,0))| \lesssim e^{-2(s_1-s)\alpha}.
\]
\end{lemma}
\begin{proof}
We can assume that $\alpha \geq 1$. Let $\theta \in [0,\frac{\pi}{2}]$, and let $\beta \geq \gamma \geq 0$ and $\phi,\phi^{\prime} \in [0,\frac{\pi}{4}]$ be determined by Proposition \ref{prop=structural_SO2}. In particular, we have
\begin{equation} \nonumber
\begin{split}
 D(\alpha,\alpha)w_{\theta,\frac{\pi}{2}-\theta}D(\alpha,\alpha)&=w_{\phi,\phi^{\prime}}D(\beta,\gamma)w_{\phi,\phi^{\prime}},\\
 D(\alpha,\alpha)w_{-\theta,\theta-\frac{\pi}{2}}D(\alpha,\alpha)&=w_{-\phi,-\phi^{\prime}}D(\beta,\gamma)w_{-\phi,-\phi^{\prime}}.
\end{split}
\end{equation}
As in the proof of Lemma \ref{lemma:eqcir}, we estimate the difference of $c$ evaluated in a generic element of the form of the first equality above and $c$ evaluated in the element for which $\theta=0$, yielding
\[
  |c(v_{t+\frac{\pi}{4}}D(2\alpha,0))-c(v_{t+\phi+\phi^{\prime}}D(\beta,\gamma))| \lesssim e^{2s\alpha} \theta^{s_1}.
\]
Similarly, using the second equality,
\[
  |c(v_{t-\frac{\pi}{4}}D(2\alpha,0))-c(v_{t-\phi-\phi^{\prime}}D(\beta,\gamma))| \lesssim e^{2s\alpha} \theta^{s_1}.
\]
Substituting $t^{\prime}=t+\frac{\pi}{4}$ in the first inequality and $t^{\prime}=t+\frac{\pi}{4}-2\phi-2\phi^\prime$ in the second one, we obtain
\[
  |c(v_{t^{\prime}}D(2\alpha,0))-c(v_{t^{\prime}+2\phi+2\phi^{\prime}-\frac{\pi}{2}}D(2\alpha,0))| \lesssim e^{2s\alpha}\theta^{s_1}.
\]
The right-hand side of this last inequality is $\lesssim e^{-2(s_1-s)\alpha}$ for $\theta \in [0,4 e^{-2\alpha}]$ (by the assumption $\alpha \geq 1$, we indeed have $4 e^{-2\alpha} \leq \frac{\pi}{2}$). In order to prove the lemma, it therefore suffices to show that, for $\alpha$ large enough, $\{ 2\phi+2\phi^{\prime}-\frac{\pi}{2} \mid \theta \in [0,4e^{-2\alpha}]\}$ contains $[-\frac{\pi}{4},0]$. By Proposition \ref{prop=structural_SO2}, it follows that $2\phi+2\phi^{\prime}-\frac{\pi}{2}$ depends continuously on $\theta$ and is equal to $0$ if $\theta = 0$, so that we only need to prove that $2\phi+2\phi^{\prime}-\frac{\pi}{2}<-\frac{\pi}{4}$ for $\theta = 4 e^{-2 \alpha}$. But for $\theta = 4 e^{-2\alpha}$, we have $\tan(2\phi) \sim 8 e^{-4\alpha}$ and $\tan(2 \phi^\prime) \sim \frac 1 2<1$ for $\alpha \to \infty$. In particular, for $\alpha$ sufficiently large, $0 \leq 2\phi+2\phi'\leq \frac{\pi}{4}$, which proves the claim.
\end{proof}
\begin{lemma} \label{lemma:eqhyp}
Let $c$ be as above, let $\beta \geq \gamma \geq 0$, and let $\alpha'$ be the non-negative solution of $\sinh^2(\alpha')=\sinh\beta\sinh\gamma$. Then there exists a $\sigma \in \mathbb{R}/2\pi\mathbb{Z}$ such that for all $t \in \mathbb{R}/2\pi\mathbb{Z}$,
\[
  |c(v_tD(\beta,\gamma))-c(v_{t+\sigma}D(\alpha',\alpha'))| \lesssim e^{s\beta-(s_2-s)\gamma}.
\]
\end{lemma}
\begin{proof}
If $0 \leq \gamma \leq 1$, the inequality holds for all $\sigma$ by the triangle inequality and the inequalities $|c(v_tD(\beta,\gamma))| \lesssim e^{s\beta}$ and $|c(v_{t+\sigma}D(\alpha',\alpha'))| \lesssim e^{s\alpha'} \leq e^{s\beta}$. We can therefore assume that $\gamma\geq 1$.

Let $\alpha$ be the non-negative solution of $\sinh\beta\sinh\gamma=\frac{1}{2}\sinh^2(\alpha)$. Clearly $\alpha' \leq \alpha$. The function $f(t) = \log(\sinh(t))$ is concave. Therefore, for a fixed value of $\beta+\gamma=x$, $\sinh \beta \sinh \gamma$ is maximal for $\beta = \gamma = \frac x 2$ and minimal for $\beta=x-1$, $\gamma=1$. This implies that $\sinh^2\alpha' \geq \sinh 1 \sinh(x-1)\geq \sinh^2\frac{x-1}{2}$ and $\sinh^2 \alpha \leq 2 \sinh^2 \frac x 2\leq \sinh^2( \frac x 2 +1)$, where the second inequalities are again consequences of the concavity of $f$. To summarize,
\begin{equation} \label{eq=alphaalphaprime}
  \frac{1}{2}(\beta+\gamma)\leq \alpha' \leq \alpha \leq \frac{1}{2}(\beta+\gamma)+1.
\end{equation}
For $t \in \R$ consider the map $c'_{t,\alpha}: h \mapsto c(v_{t} D(\alpha,0) \iota(h) D(\alpha,0))$ defined on $\mathrm{SU}(2)$. From the equality $\iota(d_\theta) = v_\theta w_{0,-2\theta}$ and the fact that $w_{0,-2\theta}$ commutes with $D(\alpha,0)$, we deduce that $D(\alpha,0) \iota(d_\theta h d_{-\theta}) D(\alpha,0) = v_{-\theta}\iota(d_\theta) D(\alpha,0) \iota(h) D(\alpha,0) \iota(d_{-\theta}) v_\theta$. Hence, for $h \in \SU$, we have $c'_{t,\alpha}(d_\theta h d_{-\theta})=c'_{t,\alpha}(h)$. Using Proposition \ref{prop:harmonic_ana_U1} and \eqref{eq:assumption_on_X_U1}, we obtain
\[
  |c(v_{t}D(\alpha,0) \iota(u_\theta) D(\alpha,0)) - c(v_{t}D(\alpha,0) \iota(u_{\frac{\pi}{2}}) D(\alpha,0))| \lesssim e^{2s\alpha}\left|\theta-\frac{\pi}{2}\right|^{s_2}.
\]
For $\theta \in [0,\frac{\pi}{2}]$ satisfying \eqref{eq=diff_beta_gamma}, it follows that $\cos \theta \leq \sqrt 2 e^{\beta - 2\alpha}$, so that the right-hand side of this inequality is $\lesssim e^{(2s - 2s_2)\alpha + s_2 \beta} \lesssim e^{s \beta - (s_2 - s)\gamma}$ by \eqref{eq=alphaalphaprime}. Also, by Proposition \ref{prop=structural_U1} and \eqref{eq=value_of_c_hyperbolas_equivariant}, we can rewrite the left-hand side, which also yields $\sigma_1,\sigma_2 \in \R/2 \pi \Z$ such that 
\[ | c(v_{t+\sigma_1}D(\beta,\gamma)) - c(v_{t+\sigma_2}D(\alpha',\alpha'))| \lesssim e^{s\beta-(s_2-s)\gamma}.\]
The result follows with $\sigma=\sigma_2-\sigma_1$.
\end{proof}
\begin{proof}[Proof of \eqref{eq:equivariant_coefficients_Sp2}]
Let $\mathcal{R}=\{(\beta,\gamma) \in \mathbb{R}^2 \mid \beta^2 + \gamma^2 \geq 2 \textrm{ and } \beta \geq \gamma \geq 0\}$. Consider the partition of $\mathcal{R}$ given by the sets $\Delta_{n}=(\mathcal{C}_{n} \cup \mathcal{H}_{n}) \cap \{(\beta,\gamma) \in \mathbb{R}^2 \mid \beta^2 + \gamma^2 \geq 2\}$, where
\[
  \mathcal{C}_{n}= \{ (\beta,\gamma) \mid \beta \geq \frac{s_1+s_2-s}{s_1}\gamma \geq 0 \textrm{ and } \sinh^2\beta+\sinh^2 \gamma \in [\sinh^2(n),\sinh^2(n+1)]\}
\]
and
\begin{align*}
  \mathcal{H}_{n}= \{ (\beta,\gamma) \mid &\frac{s_1+s_2-s}{s_1}\gamma \geq \beta \geq \gamma \geq 0 \\ &\textrm{ and } \sinh\beta\sinh \gamma \in [\sinh\beta^{(n)}\sinh\gamma^{(n)},\sinh\beta^{(n+1)}\sinh\gamma^{(n+1)}]\}.
\end{align*}
Here, $(\beta^{(n)},\gamma^{(n)})$ are the points on the line given by $\beta=\frac{s_1+s_2-s}{s_1}\gamma$ such that $\sinh^2\beta+\sinh^2\gamma=\sinh^2n$. To motivate these definitions, note that the right-hand sides of the inequalities in Lemma \ref{lemma:eqcir} and \ref{lemma:eqhyp} are equal if and only if $\beta = \frac{s_1+s_2-s}{s_1}\gamma$, and in this case they are equal to $e^{-\beta \frac{P(s)}{s_1+s_2 - s}}$. For $n \in \mathbb{N}$, let $\varepsilon(n)=e^{-n \frac{P(s)}{s_1+s_2-s}}$.

For all $(\beta_1,\gamma_1)$ and $(\beta_2,\gamma_2)$ in $\Delta_n$ with $\sinh^2\beta_1+\sinh^2\gamma_1=\sinh^2\beta_2+\sinh^2\gamma_2 \in [\sinh^2(n),\sinh^2(n+1)]$, it follows by applying Lemma \ref{lemma:eqcir} twice that there is a $\rho \in \mathbb{R}/2\pi\mathbb{Z}$ such that for all $t \in \mathbb{R}/2\pi\mathbb{Z}$, we have
\begin{equation} \label{eq=elpathcir}
  |c(v_tD(\beta_1,\gamma_1) - c(v_{t+\rho}D(\beta_2,\gamma_2))| \lesssim e^{\beta_1 s-s_1(\beta_1-\gamma_1)} + e^{\beta_2 s-s_1(\beta_2-\gamma_2)} \lesssim \varepsilon(n).
\end{equation}
The inequality $e^{\beta_i s-s_1(\beta_i-\gamma_i)} \lesssim \varepsilon(n)$ is clear if $(\beta_i,\gamma_i) \in \mathcal C_n$. Otherwise, $\beta_i+\gamma_i = \beta^{(n)}+\gamma^{(n)} + O(1)$, by the requirement that $(\beta_i,\gamma_i) \in \mathcal H_n$, and $\beta_i = n+O(1) = \beta^{(n)} + O(1)$ by $\sinh^2\beta_i+\sinh^2\gamma_i \in [\sinh^2(n),\sinh^2(n+1)]$. Hence, $\beta_i s-s_1(\beta_i-\gamma_i) = \beta^{(n)} s-s_1(\beta^{(n)}-\gamma^{(n)}) + O(1) = -n \frac{P(s)}{s_1+s_2-s} + O(1)$, which proves the inequality.

Similarly, for all $(\beta_1,\gamma_1)$ and $(\beta_2,\gamma_2)$ in $\Delta_n$ with $\sinh\beta_1\sinh\gamma_1=\sinh\beta_2\sinh\gamma_2 \in [\sinh\beta^{(n)}\sinh\gamma^{(n)},\sinh\beta^{(n+1)}\sinh\gamma^{(n+1)}]$, it follows by applying Lemma \ref{lemma:eqhyp} twice that there is a $\sigma \in \mathbb{R}/2\pi\mathbb{Z}$ such that for all $t \in \mathbb{R}/2\pi\mathbb{Z}$, we have
\begin{equation} \label{eq=elpathhyp}
  |c(v_tD(\beta_1,\gamma_1) - c(v_{t+\sigma}D(\beta_2,\gamma_2))| \lesssim e^{\beta_1 s-(s_2-s)\gamma_1} + e^{\beta_2 s-(s_2-s)\gamma_2} \lesssim \varepsilon(n).
\end{equation}
Again, this last inequality is clear if $(\beta_i,\gamma_i) \in \mathcal H_n$, and the same argument shows that otherwise $\beta_i s-(s_2-s)\gamma_i = \beta^{(n)} s-(s_2-s)\gamma^{(n)} + O(1) =-n \frac{P(s)}{s_1+s_2-s} + O(1)$.
\begin{figure}
  \center
  \includegraphics{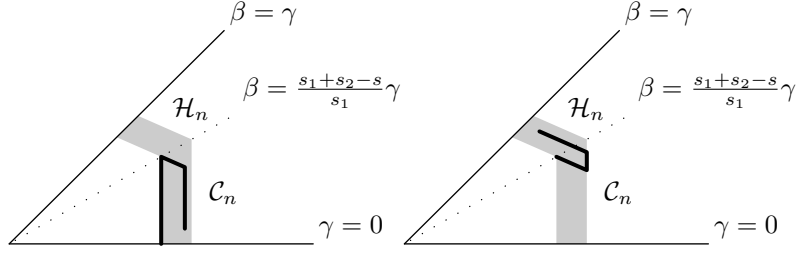}
  \caption{The three elementary paths connecting $(\beta,\gamma) \in \mathcal{C}_n$ to $(n,0)$ (left) and  $(\beta,\gamma) \in \mathcal{H}_n$ to $(\beta^{(n)},\gamma^{(n)})$ (right).}\label{picture=paths}
\end{figure}

By considering three elementary paths, meaning paths of the form of \eqref{eq=elpathcir} and \eqref{eq=elpathhyp} (see Figure \ref{picture=paths}, left), we get that for every $(\beta,\gamma) \in \mathcal{C}_{n}$, there is a $t' \in \R/2\pi \Z$ such that
\[
  |c(v_{t}D(\beta,\gamma))-c(v_{t+t^{\prime}}D(n,0))| \lesssim \varepsilon(n).
\]
Similarly by considering three elementary paths (see Figure \ref{picture=paths}, right), we get that for every $(\beta,\gamma) \in \mathcal{H}_n$, there is a $t' \in \R/2\pi \Z$ such that
\[
  |c(v_{t}D(\beta,\gamma))-c(v_{t+t^{\prime}}D(\beta^{(n)},\gamma^{(n)}))| \lesssim \varepsilon(n).
\]
Since $(\beta^{(n)},\gamma^{(n)})$ belongs also to $\mathcal C_n$, these two inequalities together imply that for every $(\beta,\gamma) \in \Delta_{n}$, there is a $t' \in \R/2\pi \Z$ such that
\[
  |c(v_{t}D(\beta,\gamma))-c(v_{t+t^{\prime}}D(n,0))| \lesssim \varepsilon(n).
\]
From this and the inequality $s_1-s \geq \frac{P(s)}{s_1+s_2-s}$, it follows by applying Lemma \ref{lemma:dependency_on_t} at most four times that for $(\beta_1,\gamma_1),(\beta_2,\gamma_2) \in \Delta_n$ and $t_1,t_2 \in \mathbb{R}/2\pi\mathbb{Z}$,
\[
  |c(v_{t_1}D(\beta_1,\gamma_1))-c(v_{t_2}D(\beta_2,\gamma_2))| \lesssim \varepsilon(n).
\]
Now, let $n,N \in \mathbb{N}$, and let $(\beta_1,\gamma_1,t_1)$ and $(\beta_2,\gamma_2,t_2)$ be two triples such that $(\beta_1,\gamma_1),(\beta_2,\gamma_2) \in \Delta_{n} \cup \ldots \cup \Delta_{n+N}$, and $t_1,t_2 \in \mathbb{R} / 2\pi\mathbb{Z}$. Using that $\frac{P(s)}{s_1+s_2-s}>0$ by our choice of $s$, by a geometric series argument that will only replace the (implicit) constant in front of the estimate, we obtain
\[
  |c(v_{t_1}D(\beta_1,\gamma_1))-c(v_{t_2}D(\beta_2,\gamma_2))| \lesssim \varepsilon(n)+\varepsilon(n+1)+\ldots+\varepsilon(n+N-1)+\varepsilon(n+N) \lesssim \varepsilon(n).
\]
From this, it follows directly that the function $c$ has a limit $c_{\xi,\eta}$ at infinity, and that
\[
  \left|\int_{\R/\pi\Z} \langle \pi(v_t g v_{-t}) \xi,\eta\rangle \frac{d t}{\pi} - c_{\xi,\eta}\right| \lesssim \varepsilon(n)
\]
for all $g \in KD(\beta,\gamma) K$ with $(\beta,\gamma) \in \cup_{l \geq n} \Delta_l$.
\end{proof}

\subsection{Non-invariant coefficients} \label{subsec=noninvc}
We now prove the second part of Theorem \ref{thm:decay_of_coefficients_Sp2}. Let $X$ be a Banach space satisfying \eqref{eq:assumption_on_X_SO2} and \eqref{eq:assumption_on_X_U1}, and let $\pi:G \rightarrow B(X)$ be a representation satisfying \eqref{eq:assumption_on_pi}. We will again assume that the restriction of $\pi$ to $K$ is isometric. Let $\pi_V:H \rightarrow B(V)$ be a unitary irreducible representation of $H$ and $m \in \Z$ with either $V$ nontrivial or $m \neq 0$. For a unit vector $\xi \in X^H$ and an $H$-equivariant linear map $q:X \to V$ with norm $1$, let $c:g \rightarrow V$ be the map defined by 
\begin{equation}
 \label{eq:def_of_nonequivariant_coeffc} c(g) = \int_{\R/\pi\Z} e^{-2imt} q (\pi(v_t g v_{-t}) \xi) \frac{d t}{\pi}.
\end{equation}
By the $H$-equivariance of $q$ it follows that
\begin{equation}\label{eq=equivariance_of_coeff_c} c(h_1v_tgv_{-t}h_2) = e^{2imt} \pi_V(h_1) c(g)\end{equation}
for all $h_1,h_2 \in H$ and $t \in \R$. Equation \eqref{eq:nonequivariant_coefficients_Sp2} will follow from the Peter-Weyl Theorem for actions of compact groups on Banach spaces (see \eqref{item2_peterweyl} in Theorem \ref{thm=peterWeyl}) and
\begin{equation}
 \label{eq=goal_in_boundind_nonequivariant coeffc} \| c(g)\| \lesssim \varepsilon(g).
\end{equation}
\begin{lemma}\label{lemma=strucS02_nonequivariant} With the notation of Proposition \ref{prop=structural_SO2}, for every $t \in \R/2\pi\Z$,
\[c(v_{t}D(\alpha,\alpha) w_{\theta,\frac{\pi}{2} - \theta}D(\alpha,\alpha)) = e^{-im(\phi+\phi')}\pi_V(d_{\frac{\phi-\phi'}{2}}) c(v_{t+\phi+\phi'}D(\beta,\gamma)).\]
\end{lemma}
\begin{proof}
This is immediate by Proposition \ref{prop=structural_SO2}, the definition of $w_{\phi,\phi'}$ and \eqref{eq=equivariance_of_coeff_c}.
\end{proof}
\begin{lemma}\label{lemma=strucU1_nonequivariant} There is a constant $C'$, depending on $V$ and $m$, such that with the notation of Proposition \ref{prop=structural_U1}, for every $t \in \R/2\pi\Z$,
\begin{align*} &\| c(v_t D(\alpha,0) \iota(u_\theta) D(\alpha,0)) \|  = \| c(v_{t+\omega_1+\omega_2} D(\beta,\gamma))\|, \\
  &\| c(v_t D(\alpha,0) \iota(u_\theta) D(\alpha,0)) - \pi_V(r_\phi) c(v_{t+\omega_1+\omega_2}D(\beta,\gamma))\| \leq C' L e^{-\gamma+s\beta}.\end{align*}
\end{lemma}
\begin{proof}
By Proposition \ref{prop=structural_U1}, the definition of $w_{\phi,\phi'}$, equation \eqref{eq=equivariance_of_coeff_c} and the fact that $\pi_V$ is unitary, the first equality is immediate.

To prove the second inequality, we use that $\pi_V$ is a Lipschitz map (it is even $C^\infty$) from $\SU$ to $B(V)$. Hence, $(\omega,\omega')\mapsto e^{-im(\omega+\omega')}\pi_V(d_{\frac{\omega-\omega'}{2}})$ is Lipschitz, say with constant $C''$, which implies that the left-hand side of the inequality in the Lemma is less than
\[ C'' (|\omega_1|+|\omega_2|) \|c(v_{t+\omega_1+\omega_2} D(\beta,\gamma))\|.\]
By \eqref{eq=structU1omega}, $|\sin(2\omega_1)|\leq (\sqrt 2 \cosh \beta)^{-1}$ and $|\sin(2\omega_2)|\leq (\sqrt 2 \cosh \gamma)^{-1}$, which implies that $|\omega_1|+|\omega_2| \lesssim e^{-\gamma}$. The Lemma follows from this inequality by the a priori bound $\|c(v_{t+\omega_1+\omega_2} D(\beta,\gamma))\| \leq \| \pi(v_{t+\omega_1+\omega_2} D(\beta,\gamma))\|_{B(X)} \leq L e^{s\beta}$.
\end{proof}
\begin{lemma}\label{lemma=Hoelder_cont_of_c_nonequivariant}
Let $c$ be as above, let $\beta\geq \gamma\geq 0$, and let $\alpha$ be the non-negative solution of $\sinh^2(2\alpha) = \sinh^2 \beta+ \sinh^2 \gamma$. Then there exist a $\rho \in \R/2\pi\Z$ and $z \in \C$ with $|z|=1$ such that for all $t \in \R/2\pi \Z$,
  \[ \| c( v_t D(\beta,\gamma)) - z c(v_{t+\rho} D(2\alpha,0))\| \lesssim e^{\beta s - s_1(\beta-\gamma)}.\]
 \end{lemma}
 \begin{proof}
For $t \in \R$, consider the map $c_{t,\alpha}:h \mapsto c(v_{t} D(\alpha,\alpha) v_{-\frac{\pi}{4}} \iota(h) D(\alpha,\alpha))$, where $h \in \mathrm{SU}(2)$. Since $D(\alpha,\alpha)$ and $v_t$ commute with $\iota(r_\theta)$ for all $\theta$, the map $c_{t,\alpha}$ satisfies $c_{t,\alpha}(r_\theta h r_{\theta'}) = \pi_V(r_\theta) c_{t,\alpha}(h)$. The characters of the abelian group $\mathrm{U}(1)$ are indexed by $n \in \Z$ and given by $r_\theta \mapsto e^{i n \theta}$. As a consequence, there is an orthonormal basis $e_1,\dots,e_d$ of $V$ and integers $m_1,\dots,m_d$ such that $\pi_V(r_\theta) e_k = e^{i m_k \theta} e_k$ for all $k$. For each $k \leq d$, the map $h \mapsto \langle c_{t,\alpha}(h),e_k\rangle$ satisfies $\langle c_{t,\alpha}(r_\theta h r_{\theta'}),e_k\rangle = e^{i m_k \theta} \langle c_{t,\alpha}(h),e_k\rangle$, and it can be written as a coefficient of the restriction to $H$ of the representation $1 \otimes \pi$ on $L^2([0,\pi),ds/\pi;X)$. Indeed,
\begin{align*} \langle c_{t,\alpha}(h),e_k\rangle =& \int_{\R/\pi\mathbb{Z}}\langle \pi(\iota(h)) \pi(D(\alpha,\alpha) v_{-t'}) \xi, (q \circ \pi(v_{t'+t} D(\alpha,\alpha) v_{-\frac{\pi}{4}}))^* e_k\rangle \frac{dt'}{\pi}\\
  =& \langle Id \otimes \pi(\iota(h)) \widetilde \xi,\widetilde \eta\rangle_{L^2([0,\pi);X),L^2([0,\pi);X^*)},
\end{align*}
where $\widetilde \xi:t' \mapsto \pi(D(\alpha,\alpha) v_{-t'}) \xi \in X$ and $\widetilde \eta:t' \mapsto (q \circ \pi(v_{t'+t} D(\alpha,\alpha) v_{-\pi/4}))^* e_k$. By \eqref{eq:assumption_on_pi} and the fact that $\|q\|\leq 1$, we have $\|\xi'\|_{L^2([0,\pi);X)} \lesssim e^{s\alpha}$ and $\|\eta'\|_{L^2([0,\pi);X^*)} \lesssim e^{s\alpha}$. For $h = d_{\frac{\pi}{4}-\theta}$ and $h' = d_{\frac{\pi}{4}}$, it follows by Proposition \ref{prop:harmonic_ana_SO2} and \eqref{eq:assumption_on_X_SO2} that
\[\|c(v_{t}D(\alpha,\alpha) w_{-\theta,\theta -\frac{\pi}{2}}D(\alpha,\alpha))-c(v_{t}D(\alpha,\alpha) w_{0,-\frac{\pi}{2}}D(\alpha,\alpha))\| \lesssim e^{2 s  \alpha} |\theta|^{s_1}.\]
For $\theta = \arctan(\sinh \gamma/\sinh \beta)$, it follows in the same way as in Lemma \ref{lemma:eqcir} that the right-hand side is $\lesssim e^{s \beta -s_1(\beta- \gamma)}$.
By Proposition \ref{prop=structural_SO2} and Lemma \ref{lemma=strucS02_nonequivariant}, we can write $c(v_{t}D(\alpha,\alpha) w_{-\theta,\theta - \frac{\pi}{2}}D(\alpha,\alpha)) = z_1 \pi_V(d_1) c(v_{t+\rho_1} D(\beta,\gamma))$ and, similarly, $c(v_{t}D(\alpha,\alpha) w_{0,-\frac{\pi}{2}}D(\alpha,\alpha)) = z_2 \pi_V(d_2) c(v_{t+\rho_2} D(2\alpha,0))$ for some $\rho_1,\rho_2 \in \R/2\pi \Z$, $d_1,d_2\in U(1)$, and $z_1,z_2 \in \C$ of modulus $1$. It follows that
  \[ \| c( v_{t+\rho_1} D(\beta,\gamma)) - \overline{z_1}z_2 \pi_V(d_1^{-1} d_2) c(v_{t+\rho_2} D(2\alpha,0))\| \lesssim e^{\beta s - s_1(\beta-\gamma)}.\]
It remains to notice that by the equality $\iota(d_\theta) D(2\alpha,0) = v_\theta D(2\alpha,0) v_{-\theta}\iota(d_\theta)$ and \eqref{eq=equivariance_of_coeff_c},
\[ \pi_V(d_\theta) c(v_{s}D(2\alpha,0)) = c(v_s \iota(d_\theta) D(2\alpha,0)) = e^{2im\theta}c(v_s D(2\alpha,0)).\]
The lemma now follows in the same way as Lemma \ref{lemma:eqcir}, with $\rho=\rho_2-\rho_1$.
\end{proof}
\begin{lemma}\label{lemma=Hoelder_cont_of_c_nonequivariant2}
  Let $c$ be as above, let $\beta \geq \gamma \geq 1$, and denote by $\alpha'$ the positive solution of $\sinh^2(\alpha') = \sinh\beta \sinh \gamma$.  There exists $\sigma \in \R/2\pi\Z$ such that for $t \in \R/2\pi \Z$,
  \[ \| c( v_t D(\beta,\gamma)) -  c(v_{t+\sigma} D(\alpha',\alpha'))\| \lesssim e^{s\beta - (s_2-s)\gamma}.\]
\end{lemma}
\begin{proof}
Let $\alpha$ be the non-negative solution of $\frac{1}{2}\sinh^2(\alpha)=\sinh\gamma\sinh\beta$. As before, \eqref{eq=alphaalphaprime} follows. For $t \in \R$ consider the map $c'_{t,\alpha}: h \mapsto c(v_{t} D(\alpha,0) \iota(h) D(\alpha,0))$ on $\mathrm{SU}(2)$. It is elementary to check that $c'_{t,\alpha}(d_\theta h d_{-\theta}) =  e^{-2 i m \theta}\pi_V(d_\theta) c'_{t,\alpha}(h)$ for all $h \in \mathrm{SU}(2)$ and $\theta \in \mathbb{R}$. Decomposing $V$ as a direct sum of characters of $\{d_\theta \mid \theta \in \R\}$, we obtain by Proposition \ref{prop:harmonic_ana_U1} that
\[\|c(v_{t}D(\alpha,0) \iota(u_\theta) D(\alpha,0)) - c(v_{t}D(\alpha,0) \iota(u_{\frac{\pi}{2}}) D(\alpha,0))\| \lesssim e^{2s \alpha} \left|\theta-\frac{\pi}{2}\right|^{s_2}.\]
As in the proof of Lemma \ref{lemma:eqhyp}, it follows that the right-hand side is $\lesssim e^{s \beta - (s_2 - s)\gamma}$.

By Lemma \ref{lemma=strucU1_nonequivariant} we can replace $c(v_{t}D(\alpha,0) \iota(u_\theta) D(\alpha,0))$ by $\pi_V(r_1) c(v_{t+\sigma_1}D(\beta,\gamma))$ up to an error of order $e^{-\gamma+s \beta}$ for some $r_1 \in \SO$ and $\sigma_1 \in \R/2\pi\Z$. We can also replace $c(v_{t}D(\alpha,0) \iota(u_{\frac{\pi}{2}}) D(\alpha,0))$ by $\pi_V(r_2) c(v_{t+\sigma_2}D(\alpha',\alpha'))$ up to an error of order $e^{(s-1)\alpha'}$ for some $r_2 \in \mathrm{SO}(2)$ and $\sigma_2 \in \mathbb{R}/2\pi\mathbb{Z}$. By \eqref{eq=alphaalphaprime}, we have $\gamma - \frac 1 2 \leq \alpha' \leq \beta+1$, and second error term is dominated by the first. By Remark \ref{rem=s1s2leq1}, $s_2 \leq s_2 - s \leq 1$, and the first error term (and therefore also the second) is dominated by $e^{s \beta - (s_2 - s)\gamma}$. Altogether, we obtain 
\[ \| \pi_V(r_1) c(v_{t+\sigma_1}D(\beta,\gamma)) -\pi_V(r_2) c(v_{t+\sigma_2}D(\alpha',\alpha'))\| \lesssim e^{s \beta - (s_2 - s)\gamma}.\]
Since $D(\alpha',\alpha')$ commutes with elements of $\iota(\SO)$, we obtain \[\pi_V(r_1^{-1} r_2) c(v_{t+\sigma_2}D(\alpha',\alpha'))=c(v_{t+\sigma_2}D(\alpha',\alpha')).\]
This proves the lemma with $\sigma=\sigma_2-\sigma_1$.
\end{proof}
\begin{proof}[Proof of \eqref{eq=goal_in_boundind_nonequivariant coeffc}] Let $\gamma_0\geq 1$, and let $\beta_0 = \frac{s_1+s_2-s}{s_1} \gamma_0$. By this choice, the right-hand sides of the inequalities of Lemmas \ref{lemma=Hoelder_cont_of_c_nonequivariant} and \ref{lemma=Hoelder_cont_of_c_nonequivariant2} with $(\beta,\gamma)=(\beta_0,\gamma_0)$ are equal to $e^{-\beta_0\frac{P(s)}{s_1+s_2-s}}$, which we simplify in this proof to $\varepsilon_0$.
Let us denote 
\[ \mathcal C= \{ (\beta,\gamma)\mid\beta \geq \frac{s_1+s_2-s}{s_1} \gamma \geq 0 \textrm{ and } \sinh^2\beta+\sinh^2 \gamma = \sinh^2\beta_0+\sinh^2 \gamma_0\}\]
and
\[ \mathcal H= \{ (\beta,\gamma)\mid\frac{s_1+s_2-s}{s_1} \gamma \geq \beta \geq \gamma \geq 0 \textrm{ and } \sinh\beta\sinh \gamma = \sinh\beta_0\sinh \gamma_0\}.\]
Denote by $\alpha$ the positive solution of $\sinh^2(2\alpha) = \sinh^2\beta_0+ \sinh^2\gamma_0$ and by $\alpha'$ the positive solution of $\sinh^2(\alpha') = \sinh\beta_0 \sinh \gamma_0$. Our goal is to prove that for every $(\beta,\gamma) \in \mathcal C\cup \mathcal H$ and every $t \in \R/2 \pi \Z$, 
\begin{equation}\label{eq=nonequivariant_coeff_small} \|c(v_t D(\beta,\gamma))\| \lesssim \varepsilon_0.\end{equation}
But Lemma \ref{lemma=Hoelder_cont_of_c_nonequivariant} implies that for $(\beta,\gamma) \in \mathcal C$,
\[ \max_{t \in \R/{2\pi \Z}} \|c(v_t D(\beta,\gamma))\| \lesssim \varepsilon_0 + \max_{t \in \mathbb{R}/2 \pi \mathbb{Z}} \|c(v_t D(2\alpha,0))\|.\]
Similarly, Lemma \ref{lemma=Hoelder_cont_of_c_nonequivariant2} implies that for $(\beta,\gamma) \in \mathcal H$,
\[ \max_{t \in \R/{2\pi \Z}} \|c(v_t D(\beta,\gamma))\| \lesssim \varepsilon_0 + \max_{t \in \mathbb{R}/2 \pi \mathbb{Z}} \|c(v_t D(\alpha',\alpha'))\|.\]
Hence, it suffices to prove \eqref{eq=nonequivariant_coeff_small} for $(\beta,\gamma)=(2\alpha,0)$ and $(\alpha',\alpha')$.

By Lemmas \ref{lemma=Hoelder_cont_of_c_nonequivariant} and \ref{lemma=Hoelder_cont_of_c_nonequivariant2} applied to $(\beta_0,\gamma_0)$, there exist $t',t''$ such that for all $t$,
\[ \|c(v_{t+t'}D(\alpha,\alpha)) - c(v_{t+t''} D(2\alpha',0))\| \lesssim \varepsilon_0.\]
The vector $c(v_{t+t'}D(\alpha,\alpha))$ belongs to the space $V_1$ of vectors invariant under $\pi_V(\SO)$, whereas $c(v_{t+t''} D(2\alpha',0))$ belongs to the space $V_2$ of vectors $x$ satisfying $\pi_V(d_\theta)x = e^{im\theta }x$ for all $\theta$.

At this point we have to distinguish two cases: if $V$ is the trivial representation, then $m\neq 0$. Necessarily, $V_2 = \{0\}$, so that \eqref{eq=nonequivariant_coeff_small} holds for $(\beta,\gamma)= (2\alpha,0)$ or $(\alpha',\alpha')$. If $V$ is not the trivial representation, note that $V_1$ and $V_2$ do not intersect ($V_1 \cap V_2$ is a space of eigenvectors for $\SO$ and $\{ d_\theta, \theta \in \R\}$, and hence of the group that they generate, which is $\SU$; as a nontrivial irreducible representation, $V$ does not have any nonzero eigenvector). Hence, there is a constant $C(V_1,V_2)>0$ such that for every $x_1 \in V_1, x_2 \in V_2$, we have $\|x_1\|+\|x_2\| \leq C(V_1,V_2) \| x_1-x_2\|$. This proves that \eqref{eq=nonequivariant_coeff_small} holds for $(\beta,\gamma)= (2\alpha,0)$ or $(\alpha',\alpha')$ and concludes the proof.
\end{proof}
\begin{rem}\label{rem=precise_value_epsilon} It follows from the proofs that the function $\varepsilon \in C_0(G)$ can be taken as 
 \begin{equation}\label{eq=precise_value_epsilon} \varepsilon(g) = \exp\left(-P(s)\max\left( \frac{\beta}{s_1+s_2-s} ,\frac{\beta+\gamma}{2s_1+s_2-s}\right)\right)\end{equation}
if $\beta \geq \gamma \geq 0$ are so that $g \in KD(\beta,\gamma)K$, and where $P(s) = s^2 - (2s_1 + s_2)s + s_1 s_2$. Indeed, with the notation of the proof of \eqref{eq:equivariant_coefficients_Sp2}, we have that $\varepsilon(g) \leq \varepsilon(n) = e^{-n \frac{P(s)}{s_1+s_2-s}}$ if $g \in K D(\beta,\gamma) K$ with $(\beta,\gamma) \in \Delta_n$. One easily checks that if $(\beta,\gamma) \in \mathcal C_n$, then $\varepsilon(g) = e^{-P(s) \frac{\beta}{s_1+s_2-s}} = e^{-n \frac{P(s)}{s_1+s_2-s} + O(1)}$. Similarly, if $(\beta,\gamma) \in \mathcal H_n$, then $\varepsilon(g) = e^{- P(s) \frac{\beta+\gamma}{2s_1+s_2-s}} = e^{-n\frac{P(s)}{s_1+s_2-s} + O(1)}$. In each case, $\varepsilon(g)$ is of the order of $\varepsilon(n)$.
\end{rem}

\section{Decay of representations of $\widetilde{\mathrm{Sp}}(2,\mathbb{R})$ on certain Banach spaces} \label{section=sp2cov}
In this section, let $\widetilde{G}=\widetilde{\mathrm{Sp}}(2,\mathbb{R})$ be the universal covering group of $G=\mathrm{Sp}(2,\mathbb{R})$ (see Section \ref{subsec:psp2cov}). Recall that $\Phi: \widetilde G \to \R$ denotes the quasi-morphism considered in Lemma \ref{lemma=phi_quasimorphism}.

Let $X$ be a Banach space satisfying \eqref{eq:assumption_on_X_SO2} and \eqref{eq:assumption_on_X_U1}. For fixed values of $s$ satisfying $0 < s < s_-(s_1,s_2)$, we consider continuous representations $\pi:\widetilde{G} \rightarrow B(X)$ for which there exist an $L>0$ such that for $\beta \geq \gamma \geq 0$,
\begin{equation} \label{eq:assumption_on_pi_bis}
  \| \pi(\widetilde D(\beta,\gamma)) \|_{B(X)} \leq L e^{s \beta}.
\end{equation}
This assumption does not say anything about the norm of $\pi(\widetilde v_t)$, but by continuity there exist $\kappa \geq 0$ and $L' \geq 0$ such that 
\begin{equation} \label{eq:assumption_on_pi_ter}
\| \pi(\widetilde v_t) \|_{B(X)} \leq L' e^{\kappa |t|}.\end{equation} The main theorem of this section, which is similar to Theorem \ref{thm:decay_of_coefficients_Sp2}, is the following.
\begin{thm} \label{thm:decay_of_coefficients_Sp2_bis}
Let $X$ be a Banach space satisfying \eqref{eq:assumption_on_X_SO2} and \eqref{eq:assumption_on_X_U1}, and let $s<s_{-}(s_1,s_2)$ and $\kappa \in \R$. There exists $\varepsilon \in C(\widetilde G)$ such that the following holds: for a representation $\pi: \widetilde G \rightarrow B(X)$ satisfying \eqref{eq:assumption_on_pi_bis} and \eqref{eq:assumption_on_pi_ter}, a unitary irreducible representation $V$ of $\SU$, and $m \in \Z$, there exists a constant $C$ such that for every $\xi \in X^H$ and $\eta \in (X^*)_V$ unit vectors,
\begin{itemize}
 \item if $V$ is the trivial representation and $m=0$, there exists $c_{\xi,\eta} \in \C$ such that
\begin{equation} \label{eq:equivariant_coefficients_Sp2_bis}
 \left|\int_{\R/\pi\Z} \langle \pi(\widetilde v_t g \widetilde v_{-t}) \xi,\eta\rangle \frac{d t}{\pi} - c_{\xi,\eta}\right| \leq C \varepsilon(g);
\end{equation}
\item otherwise, 
\begin{equation} \label{eq:nonequivariant_coefficients_Sp2_bis}
 \left|\int_{\R/\pi\Z} e^{-2imt}\langle \pi(\widetilde v_t g \widetilde v_{-t}) \xi,\eta\rangle \frac{d t}{\pi}\right| \leq C \varepsilon(g). 
\end{equation}
Moreover there is a constant $c>0$ such that $\varepsilon(g) \leq e^{ \kappa |\Phi(g)| -c \beta}$ if $g \in \widetilde K \widetilde D(\beta,\gamma) \widetilde K$ with $\beta \geq \gamma \geq 0$.
\end{itemize}
\end{thm}
Theorem \ref{thm:decay_of_coefficients_Sp2_bis} is proved in the same way as Theorem \ref{thm:decay_of_coefficients_Sp2}, with Proposition \ref{prop=structural_SO2} and \ref{prop=structural_U1} replaced by the following analogues. For $\theta,\theta' \in \R$, we define $\widetilde w_{\theta,\theta'} \in \widetilde K$ by
\begin{equation}
\widetilde w_{\theta,\theta'} = \widetilde \exp \left(\iota\begin{pmatrix}
              i\theta & 0 \\ 0 & i \theta'
             \end{pmatrix}\right) = 
  \widetilde v_{\frac{\theta+\theta'}{2}} \widetilde\iota(d_{\frac{\theta-\theta'}{2}}).
 \end{equation}
\begin{prop}\label{prop=structural_SO2_bis}
For all $\alpha > 0$ and $\theta \in [0,\frac{\pi}{2}]$, we have
\begin{align*}\widetilde D(\alpha, \alpha)\widetilde w_{\theta,\frac{\pi}{2}-\theta} \widetilde D(\alpha, \alpha) =&\widetilde w_{\phi,\phi'} \widetilde D(\beta, \gamma)\widetilde w_{\phi,\phi'}, \\ \widetilde D(\alpha, \alpha)\widetilde w_{-\theta,\theta-\frac{\pi}{2}} \widetilde D(\alpha, \alpha) =& \widetilde w_{-\phi,-\phi'} \widetilde D(\beta, \gamma)\widetilde w_{-\phi,-\phi'},\end{align*}
where $\beta, \gamma \in \R^+$, $\phi, \phi' \in [0,\frac{\pi}{4}]$ satisfy \eqref{eq=structural_SO2_D} and \eqref{eq=structural_SO2_phi}.
\end{prop}
\begin{proof} Fix $\theta \in [0,\frac{\pi}{2}]$. By Proposition \ref{prop=structural_SO2}, both equalites hold when projected onto $G$, and all terms depend continuously on $\alpha$. It is therefore enough to prove the proposition when $\alpha = 0$. In this case, the first equality therefore reduces to $\widetilde w_{\theta,\frac{\pi}{2}-\theta} = \widetilde w_{2\phi,2\phi'}$, which holds because $2\phi=\theta$ and $2\phi' = \frac{\pi}{2} - \theta$. The second equality is proved in the same way.
\end{proof}
With the same proof, we obtain the following result from Proposition \ref{prop=structural_U1}.
\begin{prop}\label{prop=structural_U1_bis}
For all $\alpha > 0$ and $\theta \in [-\frac{\pi}{2},\frac{\pi}{2}]$, we have
\[\widetilde D(\alpha,0) \widetilde \iota(u_\theta) \widetilde D(\alpha,0) = \widetilde \iota(r_\phi)\widetilde w_{\omega_1,\omega_2} \widetilde D(\beta,\gamma) \widetilde w_{\omega_1,\omega_2} \widetilde \iota(r_\phi),\]
where $\beta \geq \gamma \geq 0$, $\phi ,\omega_1,\omega_2 \in [-\frac{\pi}{4},\frac{\pi}{4}]$ are characterized by \eqref{eq=prod_beta_gamma}, \eqref{eq=diff_beta_gamma}, \eqref{eq=structU1_phi} and \eqref{eq=structU1omega}.
\end{prop}
What allows to use for $\widetilde G$ essentially the same proof as for $G$ is that in the above propositions, the value of $\Phi(g(\theta,\alpha))$ (where $g(\theta,\alpha)$ is the element of $\widetilde G$ analyzed in each proposition) remains bounded as $\alpha>0$ and $|\theta|\leq \frac{\pi}{2}$. This reflects that $\Phi$ is a quasi-morphism.
\begin{proof}[Sketch of proof of Theorem \ref{thm:decay_of_coefficients_Sp2_bis}] As for $G$, we can assume that the restriction of $\pi$ to $\widetilde H$ is isometric. For the proof of \eqref{eq:equivariant_coefficients_Sp2_bis}, denote 
\[c(g) = \int_{\R/\pi \Z} \langle \pi(\widetilde v_t g \widetilde v_{-t}) \xi, \eta\rangle \frac{dt}{\pi},\] so that $c(\widetilde h g \widetilde h' \widetilde v_t) = c(\widetilde v_t g)$ for all $t \in \R$ and $\widetilde h,\widetilde h' \in \widetilde H$. By the KAK decomposition, it is therefore enough to show \eqref{eq:equivariant_coefficients_Sp2_bis} for $g = \widetilde v_t \widetilde D(\beta,\gamma)$ for $t \in \R$ and $\beta \geq \gamma \geq 0$. Let
\[
  \widetilde{\varepsilon}(g)=\exp\left(-P(s)\max\left( \frac{\beta}{s_1+s_2-s} ,\frac{\beta+\gamma}{2s_1+s_2-s}\right)\right)
\]
for $\beta \geq \gamma \geq 0$ such that $g \in \widetilde{K}\widetilde{D}(\beta,\gamma)\widetilde{K}$ and $P(s) = s^2 - (2s_1 + s_2)s + s_1 s_2$. The same proof as for Theorem \ref{thm:decay_of_coefficients_Sp2} shows that there is a limit $c_{\xi,\eta}$ such that 
\[ |c(g) -c_{\xi,\eta}| \leq C \widetilde{\varepsilon}(g)\]
for all $g$ of the form $\widetilde v_t \widetilde D(\beta,\gamma)$ with $t \in [-\pi,\pi]$. For $t_0 \in \R$, consider $\xi^{t_0}= \pi(\widetilde v_{t_0})\xi$. It is an $H$-invariant vector of norm less that $L' e^{\kappa|t_0|}$. If we apply the preceding with $\xi$ replaced by the $H$-invariant unit vector $\xi^{t_0}/\|\xi^{t_0}\|$ we get that
\[ |c(g) -c_{\xi^{t_0},\eta}| \leq C \widetilde{\varepsilon}(g) L' e^{\kappa t_0}\]
for all $g$ of the form $\widetilde v_t \widetilde D(\beta,\gamma)$ with $t \in [t_0-\pi,t_0+\pi]$. To prove \eqref{eq:equivariant_coefficients_Sp2_bis} it remains to notice that $c_{\xi^{t_0},\eta}$ does not depend on $t_0$. Indeed if $|t_0 - t_0'|<2\pi$ the intervals $[t_0 - \pi,t_0+\pi]$ and $[t_0' - \pi,t_0'+\pi]$ intersect (say at $t$), and we get $c_{\xi^{t_0},\eta} =\lim_{\beta \to \infty} c(\widetilde v_t \widetilde D(\beta,0)) = c_{\xi^{t_0'},\eta}$. This proves  \eqref{eq:equivariant_coefficients_Sp2_bis} with $\varepsilon(g) = \widetilde \varepsilon(g) e^{\kappa |\Phi(g)|}$.

Similarly, as for \eqref{eq:nonequivariant_coefficients_Sp2} we prove \eqref{eq:nonequivariant_coefficients_Sp2_bis} for $g$ of the form $\widetilde D(\beta,\gamma)$. Replacing $\xi$ by $\pi(\widetilde v_t) \xi/\|\pi(\widetilde v_t)\xi\|$ and using the KAK decomposition, we get \eqref{eq:nonequivariant_coefficients_Sp2_bis} for arbitrary $g$, still with $\varepsilon(g) = \widetilde \varepsilon(g) e^{\kappa |\Phi(g)|}$.
\end{proof}

\section{Strong property (T)} \label{section=strongt}
In \cite{lafforguestrengthenedpropertyt}, Lafforgue proved that any connected almost $\mathbb{R}$-simple algebraic group whose Lie algebra contains a Lie subalgebra isomorphic to $\mathfrak{sl}(3,\mathbb{R})$ has property (T$^{\mathrm{strong}}_{\mathrm{Hilbert}}$). Our results allow us to generalize this to all connected higher rank simple Lie groups and to the Banach spaces in the class $\mathcal{E}_{10}$ (see Section \ref{subsec=bs} for the definition of $\mathcal{E}_{10}$), as stated in Theorem \ref{thm=maintheorem}. This section is devoted to the proof of this theorem.

It is not immediate that strong property (T) as defined in Definition \ref{def=strongT} extends from $\mathrm{SL}(3,\R)$ and $\mathrm{Sp}(2,\R)$ and its universal covering group to all connected higher rank simple Lie groups. As in \cite{lafforguestrengthenedpropertyt}, we actually show that the property (*) that we define below, which is strong property (T) with a control on the speed of convergence, holds for $\mathrm{SL}(3,\R)$ and $\mathrm{Sp}(2,\R)$ and its universal covering group. This property extends to all connected higher rank simple Lie groups.

Let $G$ be a locally compact group with a length function $\ell$. The pair ($G$,$\ell$) is said to have property (*) if there is a sequence $(m_n)$ of compactly supported symmetric measures with support contained in $\{g \in G \mid \ell(g) \leq n\}$ such that the following holds. For every Banach space $X \in \mathcal E_{10}$, there exist $\alpha,\mu>0$ such that for every continuous representation $\pi$ on $X$ satisfying $\sup_g e^{-\alpha \ell(g)}\|\pi(g)\|_{B(X)} <\infty$, there is a projection $P \in B(X)$ on the subspace $X^G$ of $G$-invariant vectors in $X$ such that
\[ \|\pi(m_n) - P\|_{B(X)} \leq e^{-\mu n}\textrm{ for $n$ large enough.}\]

We will use the following easy observation, which we state without proof.
\begin{lemma}\label{lemma=str_T_not_depend_onl} Let $G$ be a locally compact group with two length functions $\ell, \ell'$ satisfying $\ell' \leq a \ell+b$ for some $a,b > 0$, and suppose that (*) holds for $(G,\ell)$. Then (*) also holds for $(G,\ell')$.
\end{lemma}

\subsection{Case of $\mathrm{Sp}(2,\R)$ and its universal cover}
The case of $\mathrm{Sp}(2,\R)$ is a consequence of Theorem \ref{thm:decay_of_coefficients_Sp2} and of the results from \cite{delasalle1}. Indeed, Proposition 3.2 in \cite{delasalle1} and Lemma \ref{lemma=estimate_T_theta_in_schatten_classes} and \ref{lemma=estimate_S_theta_in_schatten_classes} imply that for every Banach space $X$ in $\mathcal E_{10}$, there exist constants $C_1>0,C_2>0,s_1>0,s_2>0$ such that \eqref{eq:assumption_on_X_SO2} and \eqref{eq:assumption_on_X_U1} hold. 
\begin{prop}\label{prop=strong_T_Sp2} The group $\mathrm{Sp}(2,\R)$ satisfies (*) for every length function $\ell$.
\end{prop}
\begin{proof}
Let $\ell\colon\mathrm{Sp}(2,\mathbb{R}) \rightarrow \mathbb{R}_{+}$ be the length function given by $\ell(kD(\beta,\gamma)k')=\beta$, where $\beta \geq \gamma \geq 0$ and $k,k' \in K$. Equivalently, $\ell(g) = \log \|g\|$, where the norm is the usual norm of a linear map acting on Euclidean $\R^4$. By Lemma \ref{lemma=str_T_not_depend_onl}, we only have to prove (*) for this length function because any other length function $\ell'$ satisfies $\ell' \leq a\ell +b$ for some $a,b$. We will prove (*) with 
\[ m_n(f) = \int_{H \times H \times \R/\pi \Z} f(hv_t D(n,0) v_{-t}h') dh dh' \frac{dt}{\pi}.\]
We can write $D(n,0)^{-1} = D(-n,0) = \iota(h) D(n,0) \iota(h^{-1})$ with $h = \begin{pmatrix} i & 0 \\ 0 & -i \end{pmatrix} \in \SU$. As a consequence, $H D(n,0) H = H D(n,0)^{-1} H$, and, hence, $m_n$ is a symmetric measure. 

Let $X \in \mathcal E_{10}$. As explained, there exist $C_1>0,C_2>0,s_1,s_2>0$ such that \eqref{eq:assumption_on_X_SO2} and \eqref{eq:assumption_on_X_U1} hold. With the notation of Section \ref{section=sp2}, let $\alpha<s_-(s_1,s_2)$, and let $\pi$ be a continuous representation of ${\mathrm{Sp}}(2,\R)$ on $X$ satisfying $\sup_g e^{-\alpha \ell(g)}\|\pi(g)\|_{B(X)} <\infty$. From \eqref{eq:equivariant_coefficients_Sp2} in Theorem \ref{thm:decay_of_coefficients_Sp2} and from Remark \ref{rem=precise_value_epsilon}, we obtain that $\pi(m_n)$ converges in the norm of $B(X)$ to an operator $P \in B(X)$, and 
\[\|\pi(m_n) - P\| \lesssim \exp\left( -\frac{P_{s_1,s_2}(\alpha)}{s_1+s_2-\alpha} n \right),\]
which is less than $e^{-\mu n}$ for $n$ large enough, provided that $\mu<\frac{P_{s_1,s_2}(\alpha)}{s_1+s_2-\alpha}$. It remains to show that $P$ is a projection on $X^{\mathrm{Sp}(2,\mathbb{R})}$. For this we will use \eqref{eq:nonequivariant_coefficients_Sp2}. Firstly, note that $P$ is the limit of $ \pi(m_g)$ for $g \to \infty$, where 
\[ m_g(f) = \int_{H \times H \times \R/\pi \Z} f(v_t h g h' v_{-t}) dh dh' \frac{dt}{\pi}.\]
Writing 
\[ m_{g} \ast m_{g'}(f) = \int_{H \times \R/\pi \Z} m_{g h v_t g' v_{-t}}(f) dh \frac{dt}{\pi},\]
we get that
\begin{align*} \pi(m_g) P &= \lim_{g' \to \infty} \pi(m_g \ast m_{g'}) \\ &= \lim_{g' \to \infty}\int_{H \times \R/\pi \Z} \pi(m_{g h v_u g' v_{-u}}) dh \frac{dt}{\pi} =P.\end{align*}
\[ \pi(m_g) P = \lim_{g' \to \infty} \pi(m_g \ast m_{g'}) = \lim_{g' \to \infty}\int_{H \times \R/\pi \Z} \pi(m_{g h v_u g' v_{-u}}) dh \frac{dt}{\pi} =P.\]

In particular, taking $g \to \infty$, we see that $P$ is a projection. It is clear that $X^{\mathrm{Sp}(2,\mathbb{R})} \subset P(X)$. Let us prove the converse. Let $\xi=P \xi$ be in the range of $P$ and $g \in \mathrm{Sp}(2,\mathbb{R})$. Let $V$ be an irreducible unitary representation of $H$, let $\eta \in (X^*)_V$, and let $m \in \Z$. If $m \neq 0$ or $V$ is nontrivial, we have, by  \eqref{eq:nonequivariant_coefficients_Sp2},
\begin{align*} \int_{\R/\pi\Z} &e^{-imt}\langle \pi(v_tg v_{-t})\xi,\eta\rangle \frac{dt}{\pi} \\ &= \lim_{g' \to \infty} \int_{\R/\pi \Z \times H} \int_{\R/\pi \Z} e^{-imt}  \langle \pi(v_t(g v_u h g' v_{-u}) v_{-t})\xi,\eta\rangle \frac{dt}{\pi} dh \frac{du}{\pi} = 0.\end{align*}
If $V$ is nontrivial, then this implies that all Fourier coefficients of the continuous function $t\in \R/\pi \Z \mapsto\langle \pi(v_t g v_{-t}) \xi,\eta\rangle$ vanish, meaning that $\langle \pi(v_t g v_{-t}) \xi,\eta\rangle = 0$. If $V$ is trivial, $\langle \pi(v_t g v_{-t}) \xi,\eta\rangle$ does not depend on $t$. By Theorem \ref{thm=peterWeyl}, this implies that $\pi(v_t g v_{-t}) \xi$ is an $H$-invariant vector not depending on $t$. By integrating over $H$ and $\R/\pi \Z$, we obtain
\[ \pi(g) \xi = \int_{H \times \R/\pi \Z} \pi(v_t h g v_{-t}) \xi \frac{dt}{\pi} dh = \pi(m_g) \xi.\]
We conclude that $\pi(g) \xi = \xi$ from the already proved equation $\pi(m_g) P=P$. This proves that $\xi \in X^{\mathrm{Sp}(2,\mathbb{R})}$, which finishes the proof.
\end{proof}
\begin{prop}\label{prop=strong_T_univcover} The group $\widetilde{\mathrm{Sp}}(2,\R)$ satisfies (*) for every length function $\ell$.
\end{prop}
\begin{proof}
A minor difference with the case of $\mathrm{Sp}(2,\R)$ is that we do not have a favorite length function on $\widetilde{\mathrm{Sp}}(2,\R)$, so we consider an arbitrary length function $\ell$. Then there exists a $\delta>0$ such that $\ell(\widetilde h\widetilde{v}_t\widetilde D(\lfloor \delta n \rfloor,0)\widetilde{v}_{-t}\widetilde{h}')\leq n$ for all $n$ large enough, all $\widetilde h,\widetilde h' \in \widetilde H$ and $t \in \R/\pi\Z$. Consider the measure $m_n$ on $\widetilde{\mathrm{Sp}}(2,\mathbb{R})$ given by
\[ m_n(f) = \int_{\widetilde H \times \widetilde H \times \R/\pi \Z} f(\widetilde h\widetilde v_t \widetilde D(\lfloor  \delta n \rfloor,0) \widetilde v_{-t}\widetilde h') dh dh' \frac{dt}{\pi}.\]
This measure $m_n$ is symmetric. By our choice of $\delta$, the measure $m_n$ has support in $\{g \mid \ell(g) \leq n\}$ for $n$ large enough. Using Theorem \ref{thm:decay_of_coefficients_Sp2_bis}, the proof is now essentially the same as for $\mathrm{Sp}(2,\mathbb{R})$, except that we have to be careful because the function $\varepsilon$ appearing in Theorem \ref{thm:decay_of_coefficients_Sp2_bis} does not belong to $C_0(\widetilde{\mathrm{Sp}}(2,\mathbb{R}))$. However, it will be sufficient that for each $C>0$,
\begin{equation}\label{eq=epsilon_enough_C0}
\varepsilon \in C_0\left(\{g \in \widetilde{\mathrm{Sp}}(2,\mathbb{R}) \mid |\Phi(g)|\leq C\}\right).\end{equation} Let $X \in \mathcal E_{10}$. As for $\mathrm{Sp}(2,\mathbb{R})$, there exist $\alpha, \mu>0$ such that if $\pi$ is a representation on $X$ satisfying $\sup_g e^{-\alpha \ell(g)} \|\pi(g)\|<\infty$, then $\pi(m_n)$ converges in norm to some operator $P \in B(X)$ and $\|\pi(m_n) - P\| \leq e^{-\mu n}$ for all $n$ large enough. This is because $\Phi(g) = 0$ on the support of $m_n$. More generally, by \eqref{eq=epsilon_enough_C0}, we have $P = \lim_{n \to \infty} \pi(m_{g_n})$ for every sequence $g_n$ going to infinity and such that $\sup_n \Phi(g_n)<\infty$, where
\[ m_g(f) = \int_{\widetilde H \times \widetilde H \times \R/\pi \Z} f(\widetilde{h}\widetilde{v}_t g \widetilde{v}_{-t}\widetilde{h}') d\widetilde{h} d\widetilde{h}' \frac{dt}{\pi}.\]
We show that $P$ is a projection on the invariant vectors. Using the formula
\[ m_{g} \ast m_{g'}(f) = \int_{\widetilde H \times \R/\pi \Z} m_{g \widetilde h \widetilde v_t g' \widetilde v_{-t}}(f) d\widetilde{h} \frac{du}{\pi},\]
we get, for $g'=D(\lfloor cn\rfloor,0)$ and $n \to \infty$, that $\pi(m_g) P = P$. To justify this, by \eqref{eq=epsilon_enough_C0} we have to show that for a fixed $g \in \widetilde{\mathrm{Sp}}(2,\mathbb{R})$,
\[\sup_n \sup_{\widetilde h \in \widetilde H, t \in [0,\pi]} |\Phi( g \widetilde h \widetilde v_t \widetilde D (\lfloor cn\rfloor,0) \widetilde v_{-t})|<\infty .\] This follows from \eqref{eq=phi_quasimorphism}, which implies that $|\Phi( g \widetilde h \widetilde v_t \widetilde D (\lfloor cn\rfloor,0) \widetilde v_{-t})| \leq \frac{\pi}{2} + |\Phi( g)| + |\Phi(\widetilde h \widetilde v_t \widetilde D (\lfloor cn\rfloor,0) \widetilde v_{-t})| = \frac{\pi}{2} + |\Phi( g)|$. This shows that $P$ is a projection, and the same proof as for $\textrm{Sp}(2,\R)$ shows that its range is exactly the space of invariant vectors. 
\end{proof}

\subsection{Proof of Theorem \ref{thm=maintheorem}}
To prove Theorem \ref{thm=maintheorem}, the first step is to consider the case of connected simple Lie groups locally isomorphic to $\mathrm{SL}(3,\R)$ or $\mathrm{Sp}(2,\R)$. For this we need the following lemma.
\begin{lemma}\label{lemma=str_T_passes_to_finite_extensions} Let $G$ and $H$ be locally compact groups such that $G$ is a finite extension of $H$. Then $G$ has (*) for every length function $\ell$ if and only if $H$ has (*) for every length function $\ell'$.
\end{lemma}
\begin{proof}
Let $q : G \to H$ be a surjective continuous homomorphism with finite kernel $N = \ker q$. Assume that $G$ has (*) for all $\ell$, and let $\ell'$ be a length function on $H$. Then $\ell' \circ q$ is a length function on $G$. Let $(m_n)$ be a sequence of compactly supported symmetric measures on $G$ given by property (*), and let $\widetilde m_n$ be the image of $m_n$ under $q$. It is immediate that $(\widetilde m_n)$ is a sequence of measures establishing property (*) for $(H,\ell')$. 

Conversely, assume that $H$ has property (*) for all $\ell'$. For every $f \in C_c(G)$, the function $\widetilde f(g) = \frac{1}{|N|} \sum_{n \in N} f(gn)$ is in $C_c(G/N) = C_c(H)$. If $\ell$ is a length function on $G$, the function $\ell'(h) = \max_{g \in q^{-1}(h)} \ell(g)$ is a length function on $H$, so that property (*) gives a sequence $(\widetilde m_n)$ of measures on $H$ with support contained in $\{h \mid \ell'(h) \leq n\}$. Define a measure $m_n$ on $G$ by $\int f d m_n = \int \widetilde f d \widetilde m_n$. It is symmetric and its support is contained in $\{g \mid \ell(g) \leq n\}$. If $\pi$ is a continuous representation of $G$ on a Banach space $X$, then $\pi(N) = \frac{1}{|N|}\sum_{n \in N} \pi(n)$ is a projection on $X^N$, and $\pi$ induces a representation $\widetilde \pi$ of $H$ on $X^N$, which can be extended to a continuous representation on $X$ by putting $\widetilde \pi(h) (1-\pi(N))=1-\pi(N)$. Moreover, by the definition of $m_n$, we have $\pi(m_n) = \widetilde \pi(\widetilde m_n)Q = Q \widetilde \pi(\widetilde m_n)$. From this, property (*) follows for $(G,\ell)$.
\end{proof}
\begin{lemma}\label{lemma=covering_of_sp2_have_Trenforce}
Let $R$ be a connected real Lie group with Lie algebra isomorphic to $\mathfrak{sp}_2$ or $\mathfrak{sl}_3$. Then $R$ has property (*) for every length function $\ell$.
\end{lemma}
\begin{proof}
If $R$ is isomorphic to ${\mathrm{Sp}}(2,\R)$ or $\widetilde{\mathrm{Sp}}(2,\R)$, this is Proposition \ref{prop=strong_T_Sp2} or \ref{prop=strong_T_univcover}. If $R$ is isomorphic to ${\mathrm{SL}}(3,\R)$ the lemma follows from \cite[Theorems 1.6 and 4.1]{delasalle1}.

Consider the case when the Lie algebra of $R$ is isomorphic to $\mathfrak{sp}_2$. By \cite[Propositions I.1.100 and I.1.101]{knapp}, $R$ is isomorphic to a quotient of $\widetilde{\mathrm{Sp}}(2,\R)$ by a subgroup of its center, which is $Z=\{\widetilde v_t \mid t\in \pi \Z\}$. Since all the nontrivial subgroups of $Z$ have finite index in $Z$, this implies that $R$ is either isomorphic to $\widetilde{\mathrm{Sp}}(2,\R)$, or it is, as $\mathrm{Sp}(2,\R)$, a finite extension of $\widetilde{\mathrm{Sp}}(2,\R)/Z$. In each case, the result follows, using Lemma \ref{lemma=str_T_passes_to_finite_extensions}. By the same argument and using additionally that the universal covering group of $\mathrm{SL(3,\R)}$ has finite center (of order $2$), we get in the second case that $R$ is isomorphic to a finite extension of $\mathrm{SL}(3,\R)$. This proves the lemma.
\end{proof}
\begin{proof}[Proof of Theorem \ref{thm=maintheorem}]
The proof is the same as the proof of \cite[Corollaire 4.1]{lafforguestrengthenedpropertyt}. Let $G$ be a connected simple Lie group with real rank at least two. Then $G$ has an analytic subgroup $R$ locally isomorphic to $\mathrm{SL}(3,\mathbb{R})$ or $\mathrm{Sp}(2,\mathbb{R})$ (see, e.g., \cite[Theorem 7.2]{boreltits} or \cite[Proposition I.1.6.2]{margulis}). Such a subgroup is closed, as follows from a result of Mostow (see \cite[Corollary 1]{dorofaeff}). Let $\ell$ be a length function on $G$. We will prove property (*) for $(G,\ell)$, using that $R$ has property (*) for the length function $\ell$ restricted to $R$ by Lemma \ref{lemma=covering_of_sp2_have_Trenforce}. Let $X \in \mathcal E_{10}$. Denote by $\mathfrak{g}$ the Lie algebra of $G$ (that we equip with some norm), $\mathrm{Ad}:G \rightarrow \mathrm{Aut}(\mathfrak{g})$ the adjoint representation and $\exp: \mathfrak{g} \to G$ the exponential map. Replacing $\ell(g)$ by $\ell(g) + \log \|\mathrm{Ad}(g)\|$, we can assume that $\|\mathrm{Ad}(g)\| \leq e^{\ell(g)}$ for all $g \in G$.

By Lemma \ref{lemma=covering_of_sp2_have_Trenforce}, we can find $(m_n)$, $\alpha,\mu>0$ establishing property (*) for $R$. Let $s>0$, and let $\pi$ be a continuous representation of $G$ on $X$ satisfying $\|\pi(g)\| \leq L e^{s \ell(g)}$. By Lemma \ref{lemma=covering_of_sp2_have_Trenforce}, if $s<\alpha$, then $\pi(m_n)$ converges in the norm topology of $B(X)$ to a projection $P$ onto the $R$-invariant vectors. We claim that for $s$ small enough, $P$ is a projection on the $G$-invariant vectors. For this we have to show that
\begin{equation}\label{eq=P_projection_on_XG} \pi(g) P x = P x\textrm{ for all }g \in G, x \in X.
\end{equation}
It is sufficient to show \eqref{eq=P_projection_on_XG} for all $x$ in the dense subspace of $X$ consisting of $C^\infty$ vectors, i.e., such that $g \mapsto \pi(g) x$ is $C^\infty$. In particular there is a constant $C_0$ such that $\| \pi(\exp(Z))x-x\| \leq C_0 \|Z\|_{\mathfrak{g}}$ for all $Z$ in the unit ball of $\mathfrak{g}$.

The proof relies on a variant of Mautner's Lemma. Let $a = \exp(A) \in R$ be a semisimple element such that $a \neq 1$, and decompose $\mathfrak{g} = \oplus_{\lambda \in \R} \mathfrak{g^\lambda}$ as eigenspaces of $\mathrm{ad}A$. The Lie algebra generated by $\oplus_{\lambda \neq 0}\mathfrak{g^\lambda}$ is a nonzero ideal of $\mathfrak{g}$, which is $\mathfrak{g}$ by the simplicity assumption. It is therefore sufficient to show \eqref{eq=P_projection_on_XG} when $g = \exp(Y)$ for $Y \in \mathfrak{g}^\lambda$ for some $\lambda \neq 0$. We only prove the case $\lambda>0$, the other being similar.

With these two reductions, we can now prove \eqref{eq=P_projection_on_XG}. Since $Px$ is $R$-invariant, we have for all $n \in \N$,
\begin{eqnarray*} \pi(\exp Y) Px - Px &=& \pi(a^n) \left(\pi(a^{-n} \exp Y a^n) Px - Px\right) \\&=& \pi(a^n)\left(\pi(\exp(e^{-\lambda n} Y))Px-Px \right).\end{eqnarray*}
Take $n' =\lfloor \varepsilon n\rfloor\in \N$ for some fixed $0<\varepsilon< \lambda$, and write $Px = \pi(m_{n'}) x + (P - \pi(m_{n'})) x$. By Lemma \ref{lemma=covering_of_sp2_have_Trenforce}, we have $\|(P-\pi(m_{n'}))x\| \leq e^{-\mu n'}\|x\|$ for $n$ large enough. Hence,
 \begin{multline*} \| \pi(\exp Y) Px - Px\| \leq \|\pi(a^n)\left(\pi(\exp(e^{-\lambda n} Y))\pi(m_{n'})x-\pi(m_{n'})x\right)\|\\ + 2 L e^{-\mu n' + s (n \ell(a) +1)} \|x\|\end{multline*}
for all $n$ large enough. Expanding $\pi(m_{n'})x = \int \pi(g)x d m_{n'}(g)$, we can dominate
\begin{multline*} \|\pi(a^n)\left(\pi(\exp(e^{-\lambda n} Y))\pi(m_{n'})x-\pi(m_{n'})x\right)\| \\\leq \int \| \pi(a^n g) (\pi(\exp(e^{-\lambda n} Ad(g^{-1}) Y)) - 1)x\|dm_{n'}(g).\end{multline*}
But by our assumption on $\ell$, $\|e^{-\lambda n} Ad(g^{-1}) Y\| \leq e^{-\lambda n + \ell(g^{-1})} \|Y\| \leq e^{(-\lambda + \varepsilon) n} \|Y\|$ if $g$ belongs to the support of $m_{n'}$. Since $\varepsilon<\lambda$, this is smaller than $1$ for $n$ large enough, and hence for all $g$ in the support of $m_{n'}$ we can apply the assumption that $x$ is a $C^\infty$ vector and get
\begin{eqnarray*} \| \pi(a^n g) (\pi(\exp(e^{-\lambda n} Ad(g^{-1}) Y)) - 1)x\| &\leq& C_0 \|\pi(a^n g)\| e^{(-\lambda + \varepsilon) n} \|Y\|\\ &\leq& C_0 L e^{(-\lambda + \varepsilon + s(\ell(a) +\varepsilon))n}.\end{eqnarray*} 
This implies that
\[ \| \pi(\exp Y) Px - Px\| \leq C\left(e^{(-\lambda + \varepsilon + s(\ell(a) +\varepsilon))n} + e^{(-\mu \varepsilon + s)n}\right),\]
which goes to zero as $n$ goes to infinity if $s$ is small enough. This proves the theorem.
\end{proof}

\section{Fixed point property} \label{section=fixed_point}
This section is devoted to the proof of Corollary \ref{coro=fixed_point}. Let $X$ be a Banach space, and let $\mathcal{E}$ be a class of Banach spaces containing $X \oplus \mathbb{C}$. As mentioned in the introduction, it was proved by Lafforgue that if a locally compact group $G$ has (T$^{\mathrm{strong}}_{\mathcal{E}}$), then $G$ has property ($\overline{\mathrm{F}}_X$) (and hence property (F$_X$)) \cite{lafforguestrengthenedpropertyt}.

Let $G$ be a connected simple Lie group with real rank at least $2$. Since the class $\mathcal{E}_{10}$ is stable under $X \mapsto X \oplus \C$, Corollary \ref{coro=fixed_point} for $G$ is an immediate consequence of Theorem \ref{thm=maintheorem}. However, strong property (T) is not known to pass to lattices (only to cocompact ones), and an additional argument is needed for non-cocompact lattices. This argument is based on the property of $p$-integrability, which is satisfied by the lattices under consideration. Recall from \cite{shalom} or \cite{baderfurmangelandermonod} that if $0<p<\infty$, a lattice $\Gamma$ in $G$ is $p$-integrable if it is either cocompact or for some (or equivalently any) finite generating set $S$ of $\Gamma$, there is a Borel fundamental domain $\Omega \subset G$ such that 
\[
  \int_\Omega |\chi(g^{-1}h)|_S^p dh < \infty \ \ \ \forall g \in G,
\]
where $|\cdot|_S$ is the word-length associated with $S$ and $\chi\colon G \to \Gamma$ is defined by $\chi^{-1}(1) = \Omega$ and $\chi(g\gamma^{-1})=\gamma\chi(g)$.

It is known that (F$_X$) for $\Gamma$ follows from (F$_{L^p(G/\Gamma;X)}$) for $G$ provided that $p >1$ is such that $\Gamma$ is a $p$-integrable lattice in $G$ \cite[Proposition 8.8]{baderfurmangelandermonod}. Since $\mathcal E_{10}$ is stable under $X \mapsto L^p(G/\Gamma;X)$ for any $1<p<\infty$, Corollary \ref{coro=fixed_point} follows from the following result, which Nicolas Monod kindly explained to us. A similar statement for lattices in $\widetilde{\mathrm{SL}}(2,\mathbb{R})$ that are pull-backs of lattices in $\mathrm{SL}(2,\mathbb{R})$ can be found in \cite{dastessera}. 
\begin{prop} \label{prop=pintegrable}
  Let $G$ be a connected simple Lie group with real rank at least $2$, and let $\Gamma$ be a lattice in $G$. Then $\Gamma$ is $p$-integrable for all $p < \infty$.
\end{prop}
\begin{proof} The case of $G$ having finite center coincides with the case of real linear algebraic groups, which was proved by Shalom \cite{shalom}.

If the center $Z(G)$ of $G$ is infinite, then $Z(G)$ is isomorphic to $\Z$, and we can reduce to the finite center case by using that $G$ is a central extension of $G/Z(G)$ given by a bounded $2$-cocycle \cite{guichardetwigner}. Equivalently, there exists a section $s \colon G/Z(G)  \to G$ and a finite set $A \subset Z(G)$ such that 
\begin{equation} \label{eq=bounded2cocycle}
  s(gh) s(g)^{-1}s(h)^{-1} \in A \textrm{ for all } g,h \in G/Z(G).
\end{equation}
If $\Gamma$ is a lattice in $G$, then $\Gamma Z(G)$ is discrete by \cite[Corollary 5.17]{Raghu}, and hence it is a lattice in $G$. It follows that $\Gamma$ has finite index in $\Gamma Z(G)$, and by replacing $\Gamma$ by $\Gamma Z(G)$, we can assume that $Z(G) \subset \Gamma$. Then $\Gamma/Z(G) \subset G/Z(G)$ is a lattice, so by \cite{shalom}, it is $p$-integrable for all $p<\infty$. Let $p<\infty$, and let $\Omega \subset G/Z(G)$ be a fundamental domain for $\Gamma/Z(G)$ as in the definition of $p$-integrability. For any Borel section $s \colon G/Z(G) \to G$, $s(\Omega)$ is a fundamental domain for $\Gamma$, and we claim that this fundamental domain witnesses the $p$-integrability of $\Gamma$ if $s$ satisfies \eqref{eq=bounded2cocycle}. Indeed, let $S$ be a finite symmetric generating set of $\Gamma$, and let $S'=s(S) \cup A$, where $A \subset Z(G)$ is a finite symmetric generating set of $Z(G)$ satisfying \eqref{eq=bounded2cocycle}. Then $S'$ is a finite generating set of $\Gamma$, and by \eqref{eq=bounded2cocycle} we see that $|s(\gamma_1 \gamma_2)|_{S'} \leq |s(\gamma_1)|_{S'} + |s(\gamma_2)|_{S'} +1$ for all $\gamma_1,\gamma_2 \in \Gamma/Z(G)$. This implies that $|s(\gamma)|_{S'} \leq 2 |\gamma|_{S}$ for all $\gamma \in \Gamma/Z(G)$. Also, for all $g \in G/Z(G)$, we have $s(g) s(\chi(g)) \in s(\Omega) A$, which shows that $\chi(s(g)) \in s(\chi(g)) A^{-1}$. Therefore, if $z \in Z(G)$ and $h \in \Omega$ is arbitrary, we have $(zs(g))^{-1}s(h) \in z^{-1} s(g^{-1} h) A$, and $\chi((z s(g))^{-1}s(h)) \in s(\chi(g^{-1}h)) z^{-1} A A^{-1}$, which implies that
\[
  |\chi(z s(g))^{-1}s(h)|_{S'} \leq 2 |\chi(g^{-1}h)|_S + |z|_S + 2.
\]
This implies that $\int_{\Omega} |\chi(z s(g))^{-1}s(h)|_{S'}^p dh <\infty$ for all $z \in Z(G)$ and $g \in G/Z(G)$. This concludes the proof because every element of $G$ can be written in this way.
\end{proof}

\appendix

\section{On an inequality for Jacobi polynomials} \label{section=p10}
In this appendix, we prove, using the notation of Section \ref{section=harmonic_analysis_SU2}, the following theorem.
\begin{thm}\label{thm=Hoelder_continuity_in_Schatten_classes}
 For $q>10$, there is a constant $C_q$ such that for all $\theta_1,\theta_2 \in \R$
 \[ \| S_{\theta_1} - S_{\theta_2} \|_{S^q} \leq C_q |\theta_1 - \theta_2|^{\frac{1}{4} - \frac{5}{2q}}.\]
Moreover $C_q \leq C(q-10)^{-\frac{1}{q}}$ for some universal constant $C$.
\end{thm}
The proof of this theorem is by computation: the operator $S_{\theta_1} - S_{\theta_2}$ can be explicitly diagonalized. Its eigenvalues are obtained from the spherical functions for the Gelfand pair $(\mathrm{U}(2),\mathrm{U}(1))$, which are the so-called disc polynomials (see \cite{koornwinder}). The proof relies on some careful estimates of the value of these polynomials at $0$. This is exactly the strategy of proof that was already applied in \cite{delaat1} and \cite{haagerupdelaat2}, using the estimates for the Jacobi polynomials obtained in \cite{HS}. Our only contribution is a slight improvement of the results of \cite{HS} (see Lemma \ref{lemma=dominate_integral} and the preceding remark).

Recall that the irreducible representations of $\SU$ are indexed by the non-negative half-integers $\ell = 0,\frac{1}{2},1,\frac{3}{2},\dots$. The corresponding irreducible representations $\pi_\ell$ on the complex vector spaces $\mathcal{H}_\ell$ of homogeneous polynomials of degree $2\ell$ in two complex variables $z_1,z_2$ are given by
\[
  \pi_\ell(u)P(z_1,z_2) = P\left(a z_1 + c z_2,b z_1 + d z_2\right), \qquad u = \begin{pmatrix} a &b \\c&d\end{pmatrix}.
\]
If $\mathcal{H}_\ell$ is equipped with a Hilbert space structure that makes $\pi_\ell$ into a unitary representation, then the family $\{z_1^m z_2^n \mid m+n = 2\ell\}$ is an orthogonal family. It is convenient to index this family by $p=-\ell,-\ell+1,\dots,\ell$ and to denote the polynomial $z_1^{\ell-p}z_2^{\ell+p}$ by $h_p^\ell$. For $\ell,p,p'$, let $f^\ell_{p,p'}:\SU \to \C$ be the matrix element given by $f^\ell_{p,p'}(g) = \langle \pi_\ell(g) h_{p'}^\ell, h_p^\ell\rangle$. The matrix formed by these elements is called the Wigner $D$-matrix. By the Peter-Weyl Theorem and the orthogonality of the family $\{h_p^\ell \mid p=-\ell,-\ell+1,\dots,\ell\}$ in $\mathcal{H}_\ell$, the family of functions $f^\ell_{p,p'}$ for $\ell \in \N/2$ and $p,p'=-\ell,-\ell+1,\dots,\ell$ form an orthogonal basis of $L^2(\SU)$. It turns out that the operators $S_\theta$ are all diagonal in this basis.
\begin{lemma}
For $\ell \in \N/2$ and $p=-\ell,-\ell+1,\dots,\ell$,
\[ S_\theta f^\ell_{p,p'} = e^{2i p \theta} c_p^\ell f^\ell_{p,p'}\]
where, for every $r>0$, 
\begin{equation}\label{eq=integral_formula_for_Jacobi_pol} c_p^\ell = 2^{-\ell}\int_0^{2\pi} (1+r^{-1}e^{-i\varphi})^{\ell -p}(1-r e^{i\varphi})^{\ell+p} \frac{d \varphi}{2\pi}.\end{equation}\end{lemma}
\begin{proof}
By the definition of $S_\theta$ and $f^\ell_{p,p'}$ we can write $S_\theta f^\ell_{p,p'}(g) = \langle \pi_\ell(g)h_{p'}^\ell, I_p^\ell(\theta)\rangle$, where $I_p^\ell(\theta) = \int_0^{2\pi} \pi_\ell(d_\varphi u_{\theta} d_{-\varphi})h_{p}^\ell \frac{d\varphi}{2\pi}$. By expanding
\[ \pi_\ell(d_\varphi u_{\theta} d_{-\varphi})h_{p}^\ell = \left(\frac{e^{i\theta}z_1 + e^{-2i\varphi} z_2}{\sqrt 2}\right)^{l-p} \left(\frac{- e^{2i\varphi} z_1+e^{-i\theta}z_2}{\sqrt 2}\right)^{\ell+p}\] in the basis of $(h_p^\ell)_{p = -\ell, \dots,\ell}$, a small computation yields the existence of $c_p^\ell \in \R$ such that $I_p^\ell(\theta) = e^{-2i p\theta} c_p^\ell h_p^\ell$. By substituting $\theta = 0$, $z_1=r$ and $z_2=1$, we get
\[ \int_0^{2\pi} \left(\frac{r + e^{-2i\varphi} }{\sqrt 2}\right)^{l-p} \left(\frac{- e^{2i\varphi}r+e^{-2i\theta}}{\sqrt 2}\right)^{\ell+p} = c_p^\ell r^{\ell - p}.\]
This is equation \eqref{eq=integral_formula_for_Jacobi_pol} after the change of variable $2\varphi \to \varphi$. This also implies that $S_\theta f^\ell_{p,p'} =  e^{2 i p \theta} c_p^\ell f^\ell_{p,p'}$.
\end{proof}
It follows from this Lemma and the above description of the unitary dual of $\SU$ that for $\theta_1,\theta_2 \in \R$ and $q >0$,
\begin{equation}\label{eq=expansion_Schatten_norm_Stheta} \| S_{\theta_1} - S_{\theta_2} \|_{S^q}^q = \sum_{\ell}\sum_{p=-\ell,\dots,\ell} (2\ell+1) \left| (e^{2i p\theta_1}-e^{2i p \theta_2}) c_p^\ell\right|^q.\end{equation}
An upper bound for $\| S_{\theta_1} - S_{\theta_2} \|_{S^q}$ will follow from an upper bound on $c_p^\ell$. For $p=p'=\ell = \frac{1}{2}$, we have $c_p^\ell = 2^{-\frac{1}{2}}$ and $S_\theta f_{p,p'}^\ell = 2^{-\frac{1}{2}} e^{i\theta} f_{p,p'}^\ell$. This implies the following ``obvious'' lower bound, used in the proof of Proposition \ref{prop:harmonic_ana_U1}:\begin{equation}\label{eq=lower_inequality_S_theta} \| S_{\theta_1} - S_{\theta_2} \|_{B(L^2(\SU))} \geq 2^{-\frac{1}{2}} |e^{i\theta_1} - e^{i \theta_2}|.
\end{equation}
\begin{rem}\label{rem}
For $p \geq 0$, the constant $c_p^\ell$ is, up to a factor $(-2)^{\ell-p}$, the value of the Jacobi polynomial $P_{\ell - p}^{(0,2p)}$ at $0$. For $p < 0$, we have $c_p^\ell = c_{-p}^\ell$. In \cite{HS}, it was proved that $|c_p^\ell| \leq C (1+l)^{-\frac{1}{4}}$. We will improve this estimate in Proposition \ref{prop=cpl}.
\end{rem}
The key lemma for the improvement given by Theorem \ref{thm=Hoelder_continuity_in_Schatten_classes} is a slight modification of \cite[Lemma 3.6]{HS}.
\begin{lemma}\label{lemma=dominate_integral}
There is a constant $C$ such that for every $u,v \in \R$,
\[ \int_0^\pi e^{-(u-v\cos s)^2} \frac{ds}{\pi} \leq \frac{C}{\sqrt{(|u+v|+1)(|u - v|+1)}}.\]
\end{lemma}
\begin{proof}
By symmetry we can assume that $u,v \geq 0$. We can also assume $v \geq 1$, since for $v \in [0,1]$ and $u \geq 0$, the term $\int_0^\pi e^{-(u-v\cos s)^2} \frac{ds}{\pi}$ is less than $e^{1-(u-1)^2}$, which is less than $\frac{C}{\sqrt{(|u+v|+1)(|u - v|+1)}}$ for some $C$.

Firstly, assume that $0 \leq u \leq v$. The case of $|u-v|\leq 1$ was already covered in \cite{HS}. Hence, we additionally assume that $|u-v| \geq 1$. Let $\sigma \in [0,\frac{\pi}{2}]$ be such that
$\cos{\sigma}=\frac{u}{v}$. Note that $\sigma^2\geq 2\frac{v-u}{v}$, since $\cos \sigma \geq 1 - \frac{\sigma^2}{2}$. Then, as in \cite{HS},
\[ u-v\cos{s}=v(\cos\sigma-\cos s)=
2v\sin\left(\frac{s+\sigma}{2}\right)\sin\left(\frac{s-\sigma}{2}\right).\] 
Using the inequality $|\sin t| \geq \frac{2 \sqrt 2}{3\pi}|t|$ for $|t| \leq \frac{3\pi}{4}$, we conclude that
\[|u-v \cos s| \geq \frac{4 v}{9 \pi^2}(s+\sigma)|s-\sigma| \geq \frac{4 v\sigma}{9 \pi^2}|s-\sigma|.\]

Hence,
\begin{equation*}
\begin{split}
\frac{1}{\pi} \int_{0}^{\pi} e^{-(u-v\cos{s})^2} ds &\leq \frac{1}{\pi}
\int_{0}^{\pi} e^{-(\frac{4}{9} v \sigma \pi^{-2})^2 (s-\sigma)^2} ds
\\
& \leq \frac{1}{\pi}
\int_\R e^{-(\frac{4}{9} v \sigma \pi^{-2})^2 (s-\sigma)^2} ds
\\
&= \frac{9 \pi^{3/2}}{4 v \sigma} \leq \frac{10}{\sqrt{v(v-u)}}
\end{split}
\end{equation*}
where we used that $\sigma^2 \geq 2\frac{v-u}{v}$ and $\frac{9 \pi^{3/2}}{4 \sqrt 2} \leq 10$. By the assumptions that $v \geq v-u \geq 1$, we have $\frac{10}{\sqrt{v(v-u)}} \leq \frac{C}{\sqrt{(|u+v|+1)(|u-v|+1)}}$ for some universal constant $C$.

Now, assume that $0 \leq v \leq u$. Then $u-v\cos{s} = u-v + 2v\sin^2\left(\frac{s}{2}\right) \geq u-v + \frac{2}{\pi^2} v s^2$. Hence,
\begin{equation*}
\begin{split}
\frac{1}{\pi} \int_{0}^{\pi} e^{-(u-v\cos{s})^2}ds &\leq \frac{1}{\pi}
\int_{0}^{\pi} e^{-(u-v)^2 - \frac{4 v^2}{\pi^4} s^4} ds \\
&\leq \frac{1}{\pi} \int_0^{\infty} e^{-(u-v)^2 - \frac{4 v^2}{\pi^4} s^4} ds
\leq \frac{e^{-(u-v)^2}}{\sqrt{2v}}
\end{split}
\end{equation*}
using $\int_0^{\infty}e^{-t^4}\,dt\leq1$. Finally by our assumption that $v \geq 1$, we have $\frac{e^{-(u-v)^2}}{\sqrt{2v}} \leq \frac{C}{\sqrt{(|u+v|+1)(|u - v|+1)}}$ for some universal constant $C$.
\end{proof}
We now proceed to the crucial estimate.
\begin{prop} \label{prop=cpl}
There is a constant $C$ such that for every $\ell \in \N/2$, $p = -\ell,\ell+1,\dots,\ell$,
\begin{equation}\label{eq=domination_of_c} |c_p^\ell| \leq C \min\left((1+\ell)^{-\frac{1}{4}},\left| |p| - \frac{\ell}{\sqrt 2}\right|^{-\frac{1}{2}}\right).
\end{equation}
In words, we get strictly better estimates than $|c_p^\ell| \leq C(1+\ell)^{-\frac{1}{4}}$ except on an interval of size $\sqrt{1+\ell}$ around $\pm \ell/\sqrt 2$.
\end{prop}
The proof is the same as the proof of \cite[Theorem 1.1]{HS}, except that we use the result of Lemma \ref{lemma=dominate_integral} instead of \cite[Lemma 3.6]{HS}.
\begin{proof}[Proof of Theorem \ref{thm=Hoelder_continuity_in_Schatten_classes}]
We will use in the proofs that for all $\alpha>1$, $u \in \R$ and $x>0$,
\begin{equation}\label{eq=reste_serie} \sum_{k \in (u+\Z) \cap (x,\infty)} k^{-\alpha}  \leq \frac{\alpha}{\alpha - 1} x^{1-\alpha}.\end{equation}
This follows from the computation $\sum_{k \in (u+\Z) \cap (x,\infty)} k^{-\alpha} \leq x^{-\alpha} + \int_x^{\infty} y^{-\alpha}dy \leq \frac{\alpha}{\alpha-1}x^{1-\alpha}$ for $x\geq1$. For $0<x<1$, we have $\sum_{k \in (u+\Z) \cap (x,\infty)} k^{-\alpha} \leq 1 + \int_x^{\infty} y^{-\alpha}dy \leq \frac{\alpha}{\alpha-1}x^{1-\alpha}$.

\textbf{In the rest of the proof we assume} $q>10$. We will denote $A \lessapprox B$ if there exists a universal constant $K$ such that $A \leq K^q B$.

If we denote $\alpha_p = |e^{2i p\theta_1}-e^{2i p \theta_2}|$, then \eqref{eq=expansion_Schatten_norm_Stheta} becomes \[\| S_{\theta_1} - S_{\theta_2} \|_{S^q}^q = \sum_{p \in \frac 1 2 \Z} \alpha_p^q \sum_{\ell \in |p|+ \N} (2\ell+1) |c_p^\ell|^q.\]
If $p=0$, then $\alpha_p=0$. For $|p|=\frac{1}{2}$, using \eqref{eq=reste_serie}, we obtain
\[   \sum_{\ell \in \frac{1}{2} + \mathbb{N}} (2\ell+1)|c_{\frac{1}{2}}^{\ell}|^q \leq 2C^q\sum_{\ell \in \frac{1}{2} + \mathbb{N}} (\ell+1)^{1-\frac{q}{4}} \lessapprox 1.\]
For fixed $p$ different from $0$ and $\pm \frac{1}{2}$, we decompose the sum $\sum_{\ell \in |p|+ \N}$ as 
\[ \sum_{|p| \leq \ell < \sqrt 2 |p| - \sqrt{|p|}} + \sum_{\sqrt{2} |p| - \sqrt{|p|} \leq \ell \leq \sqrt{2} |p| + \sqrt{|p|}+1} + \sum_{\sqrt{2} |p| + \sqrt{|p|}+1 < \ell}.\]
For the first sum, by \eqref{eq=domination_of_c}, we obtain $(2\ell+1) |c_p^\ell|^q \leq 2 \sqrt{2} |p| C^q (|p| - \frac{\ell}{\sqrt{2}})^{-\frac{q}{2}}$, so that with the change of variable $k=\sqrt{2} |p|- \ell \in  (\sqrt{2}-1)|p| + \Z$, we obtain by \eqref{eq=reste_serie},
\[  \sum_{|p| \leq \ell < \sqrt 2 |p| - \sqrt{|p|}} (2\ell+1) |c_p^\ell|^q  \leq 2 \sqrt{2} |p| C^q \sum_{k > \sqrt{|p|}} \left(\frac{k}{\sqrt 2}\right)^{-\frac{q}{2}} \lessapprox |p|^{\frac{3}{2}-\frac{q}{4}}.\]

The second sum has at most $2 \sqrt{|p|}+1$ terms. By \eqref{eq=domination_of_c}, for each of these terms, we have $(2\ell+1) |c_p^\ell|^q \leq 2 C^q (1+\ell)^{1-\frac{q}{4}} \lessapprox |p|^{1 - \frac{q}{4}}$ so that we get 
\[ \sum_{\sqrt{2} |p| - \sqrt{|p|} \leq \ell \leq \sqrt{2} |p| + \sqrt{|p|}} (2\ell+1) |c_p^\ell|^q \lessapprox |p|^{\frac 3 2 - \frac q 4}.\]

For the third sum, we obtain $(2\ell+1) |c_p^\ell|^q \leq C^q (2\ell+1) (\frac{\ell}{\sqrt{2}} - |p|)^{-\frac{q}{2}}$. With the change of variable $k = \ell - \sqrt 2 |p| \in (1 - \sqrt 2)|p|+\Z \cap (\sqrt{|p|},\infty)$, this becomes $(2\ell+1) |c_p^\ell|^q \leq C^q (2k + 2 \sqrt 2 |p|+1) \left(\frac{k}{\sqrt{2}}\right)^{-\frac{q}{2}} \lessapprox (k+|p|) k^{-\frac q 2}$. Hence, by \eqref{eq=reste_serie},
 \[ \begin{split} \sum_{\sqrt{2} |p| + \sqrt{|p|} \leq \ell}  (2\ell+1) |c_p^\ell|^q &\lessapprox \sum_{k \in (1 - \sqrt 2)|p|+\Z \cap (\sqrt{|p|},\infty)} k^{1-\frac q 2} + |p|k^{-\frac q 2}
\\ 
& \lessapprox \left(\sqrt{|p|}^{2-\frac q 2} + |p| \sqrt{|p|}^{1-\frac q 2}\right) \lessapprox |p|^{\frac 3 2 - \frac q 4}.\end{split}\]

Adding the three sums above, we obtain for $|p| \geq 1$,
\[ \sum_{\ell \in |p|+ \N} (2\ell+1) |c_p^\ell|^q \lessapprox |p|^{\frac 3 2 - \frac q 4}.\]

All together, using $\alpha_{-p}=\alpha_{p}$, we obtain
\[ \| S_{\theta_1} - S_{\theta_2} \|_{S^q}^q \lessapprox \sum_{p \in \frac 1 2 \Z} \alpha_p^q |p|^{\frac{3}{2} - \frac{q}{4}} \lessapprox C^q \sum_{n \geq 1} \alpha_{\frac{n}{2}}^q n^{\frac{3}{2} - \frac{q}{4}}.\]

Denote $\varepsilon = | e^{2i\theta_1} - e^{2i\theta_2}| \in [0,2]$. We use the inequality $\alpha_p = |e^{2i p\theta_1}-e^{2 i p \theta_2}| \leq \min(2,p \varepsilon)$. Let $n_0$ be the least $n$ such that $n \varepsilon \leq 4$. Then
\[ \sum_{1 \leq n\leq n_0} \alpha_{\frac{n}{2}}^q  n^{\frac{3}{2} - \frac{q}{4}} \leq 2^{-q} \varepsilon^q  \sum_{1 \leq n\leq n_0} n^{\frac{3}{2} +\frac{3 q}{4}} \lessapprox \varepsilon^q (1+n_0)^{\frac{5}{2} +\frac{3 q}{4}}\lessapprox \varepsilon^{\frac{q}{4} - \frac{5}{2}},\]
where the last inequality holds because $1+n_0 \leq 1+\frac{4}{\varepsilon} \leq \frac{6}{\varepsilon}$. In the same way using that $n_0 \geq \frac{4}{\varepsilon} - 1 \geq \frac{1}{\varepsilon}$
\[ \begin{split} \sum_{n_0<n} \alpha_{\frac{n}{2}}^q  n^{\frac{3}{2} - \frac{q}{4}} & \leq 2^{q}  \sum_{n_0<n} n^{\frac{3}{2} -\frac{q}{4}}
\lessapprox \frac{n_0^{\frac{5}{2} -\frac{ q}{4}}}{q-10}
\lessapprox \frac{\varepsilon^{\frac{q}{4}-\frac{5}{2} }}{q-10}.
\end{split}\]

Together we get
\[ \| S_{\theta_1} - S_{\theta_2} \|_{S^q}^q \lessapprox \frac{q}{q-10} \varepsilon^{\frac{q}{4}-\frac{5}{2} }.\]
It remains to use that $\varepsilon \leq 2 |\theta_1 - \theta_2|$.
\end{proof}


\begin{thebibliography}{99}
\bibitem[BFGM07]{baderfurmangelandermonod}
U.~Bader, A.~Furman, T.~Gelander and N.~Monod, \emph{Property (T) and rigidity for actions on Banach spaces}, Acta Math.~\textbf{198} (2007), 57--105.

\bibitem[BG92]{BargeGhys}
J.~Barge and {\'E}.~Ghys, \emph{Cocycles d'Euler et de Maslov}, Math.~Ann.~\textbf{294} (1992), 235--265.

\bibitem[BHV08]{bekkadelaharpevalette}
B.~Bekka, P.~de la Harpe and A.~Valette, \emph{Kazhdan's property (T)}, Cambridge University Press, Cambridge, 2008.

\bibitem[BT65]{boreltits}
A.~Borel and J.~Tits, \emph{Groupes r\'eductifs}, Inst.~Hautes~\'Etudes~Sci.~Publ.~Math., no.~27 (1965), 55--150.

\bibitem[DT14]{dastessera}
K.~Das and R.~Tessera, \emph{Integrable measure equivalence and the central extension of surface groups}, preprint (2014), arXiv:1405.2667.

\bibitem[Dor96]{dorofaeff}
B.~Dorofaeff, \emph{Weak amenability and semidirect products in simple {Lie} groups}, Math.~Ann.~\textbf{306} (1996), 737--742.

\bibitem[Dup79]{dupont}
J.~Dupont, \emph{Bounds for characteristic numbers of flat bundles}, In: Algebraic topology, Aarhus 1978 (Proc.~Sympos., Univ.~Aarhus, Aarhus, 1978), pp.~109--119, Lecture Notes in Math., vol.~763, Springer, Berlin, 1979.
  
\bibitem[DG78]{dupontguichardet}
J.~Dupont and A.~Guichardet, \emph{\`{A} propos de l'article: ``{S}ur la cohomologie r\'eelle des groupes de {L}ie simples r\'eels''}, Ann.~Sci.~\'Ecole Norm.~Sup.~(4), \textbf{11} (1978), 293--295.

\bibitem[GW78]{guichardetwigner}
A.~Guichardet and D.~Wigner, \emph{Sur la cohomologie r\'eelle des groupes de {L}ie simples r\'eels}, Ann.~Sci.~\'Ecole Norm.~Sup.~(4), \textbf{11} (1978), 277--292.

\bibitem[HdL13a]{haagerupdelaat1}
U.~Haagerup and T.~de Laat, \emph{Simple Lie groups without the Approximation Property}, Duke Math.~J.~\textbf{162} (2013), 925--964.

\bibitem[HdL13b]{haagerupdelaat2}
U.~Haagerup and T.~de Laat, \emph{Simple Lie groups without the Approximation Property II}, to appear in Trans.~Amer.~Math.~Soc., preprint (2013), arXiv:1307.2526.

\bibitem[HS13]{HS}
U.~Haagerup and H.~Schlichtkrull, \emph{Inequalities for Jacobi polynomials}, Ramanujan J.~\textbf{33} (2014), 227--246.

\bibitem[Hel78]{helgasonlie}
S.~Helgason, \emph{Differential Geometry, Lie Groups and Symmetric Spaces}, Academic Press, New York, 1978.

\bibitem[Kaz67]{kazhdan}
D.A.~Kazhdan, \emph{Connection of the dual space of a group with the structure of its closed subgroups}, Funk.~Anal.~Appl.~\textbf{1} (1967) 63--65. 

\bibitem[Kna02]{knapp}
A.W.~Knapp, \emph{Lie Groups beyond an Introduction}, Birkh\"auser Boston Inc., Boston, MA, second edition, 2002.

\bibitem[Koo72]{koornwinder}
T.H.~Koornwinder, \emph{The addition formula for Jacobi polynomials II. The Laplace type integral representation and the product formula}, Report TW 133/72, Mathematisch Centrum, Amsterdam, 1972.

\bibitem[dL13]{delaat1}
T.~de Laat, \emph{Approximation properties for noncommutative $L^p$-spaces associated with lattices in Lie groups}, J.~Funct.~Anal.~\textbf{264} (2013), 2300--2322.

\bibitem[Laf08]{lafforguestrengthenedpropertyt}
V.~Lafforgue, \emph{Un renforcement de la propri\'et\'e {(T)}}, Duke Math.~J.~\textbf{143} (2008), 559--602.

\bibitem[Laf09]{lafforguefourier}
V.~Lafforgue, \emph{Propri\'et\'e ({T}) renforc\'ee banachique et transformation de {F}ourier rapide}, J.~Topol.~Anal.~\textbf{1} (2009), 191--206.
  
\bibitem[LdlS11]{lafforguedelasalle}
V.~Lafforgue and M.~de la Salle, \emph{Non commutative {$L^p$} spaces without the completely bounded approximation property}, Duke.~Math.~J.~\textbf{160} (2011), 71--116.
  
\bibitem[Lia13]{liao}
B.~Liao, \emph{Strong banach property {(T)} for simple algebraic groups of higher rank}, J.~Topol.~Anal.~\textbf{6} (2014), 75--105.

\bibitem[Lub10]{lubotzky}
A.~Lubotzky, \emph{Discrete Groups, Expanding Graphs and Invariant Measures}, Birkh\"auser Verlag, Basel, 2010.

\bibitem[Mar91]{margulis}
G.~Margulis, \emph{{Discrete Subgroups of Semisimple Lie Groups}}, Springer-Verlag, Berlin, 1991.

\bibitem[Mau03]{maurey}
B.~Maurey, \emph{Type, cotype and $K$-convexity}, In:~Handbook of the geometry of Banach spaces, vol.~2, pp.~1299--1332, North-Holland, Amsterdam, 2003.

\bibitem[Pis10]{pisiermemams}
G.~Pisier, \emph{Complex interpolation between Hilbert, Banach and operator spaces}, Mem.~Amer.~Math.~Soc.~\textbf{208} (2010), no.~978.

\bibitem[PX87]{pisierxu}
G.~Pisier and Q.~Xu, \emph{Random series in the real interpolation spaces between the spaces {$v_p$}}, In:~Geometrical aspects of functional analysis (1985/86), Lecture Notes in Math., vol.~1267, pp.~185--209. Springer, Berlin, 1987.

\bibitem[Rag72]{Raghu}
M.S.~Raghunathan, \emph{Discrete subgroups of Lie groups}, Springer-Verlag, Berlin, 1972.

\bibitem[Raw12]{rawnsley}
J.~Rawnsley, \emph{On the universal covering group of the real symplectic group}, J.~Geom.~Phys.~\textbf{62} (2012), 2044--2058.

\bibitem[dlS13]{delasalle1}
M.~de la Salle, \emph{Towards Banach space strong property (T) for $\mathrm{SL}(3,\mathbb{R})$}, to appear in  Israel J.~Math., preprint (2013), arXiv:1307.2475.

\bibitem[Sha00]{shalom}
Y.~Shalom, \emph{Rigidity of commensurators and irreducible lattices}, Invent.~Math.~\textbf{141} (2000), 1--54.
\end{thebibliography}
\end{document}